\documentclass[12pt]{article}
\newcommand{\average}{-\!\!\!\!\!\!\int}

\newcommand{\comment}[1]{}

\renewcommand{\div}{{\mathop{\rm\,div\,}\nolimits}}

\newcommand{\dist}{\mathop{\rm dist}\nolimits}

\providecommand{\qed}{\vrule height 6pt depth 0pt width 3 pt}

\newcommand{\reals}{{\bf R}}

\newcommand{\sphere}{{\bf S}}

\newcommand{\BibTeX}{{\rm B\kern-.05em{\sc i\kern-.025em b}\kern-.08em     
    T\kern-.1667em\lower.7ex\hbox{E}\kern-.125emX}}

\usepackage{hyperref}

\newcommand{\note}[1]{}
\renewcommand\marginpar[1]{}

\usepackage{amsmath}
\usepackage{amsmath} 
\usepackage{amssymb}

\newcommand{\tent}[2]{{\cal N} ^ {#1}(#2)}
\newcommand{\ddb}[2]{{\tilde S}_{#2}({#1})}
\newcommand{\ssb}[2]{S_{#2}({#1})}
\newcommand{\hlmax}{{\cal M}}

\newcommand{\sobolev}[2]{  W ^ {{#2},{#1}}}
\newcommand{\ball}[2]{B_{#2}(#1)}

\newcommand{\locdom}[2]{{\Omega_{#2}(#1)}}
\newcommand{\sball}[2]{{\Delta_{#2}(#1)}}
\newcommand{\cyl}[2]{Z_{#2}(#1)}
\newcommand{\avg}[3]{ \bar #1 _{#2,#3}}

\newcommand{\nontan}[1]{#1^*}

\newenvironment{proof}[1][Proof]{\begin{trivlist}\item[\hskip \labelsep
{\it #1. }]}{\hfill \qed \goodbreak \end{trivlist}}   
   
\numberwithin{equation}{section} 
\newtheorem{theorem}[equation]{Theorem}

\newtheorem{lemma}[equation]{Lemma}

\renewcommand{\theequation}{\arabic{section}.\arabic{equation}}

\usepackage{amssymb}

\begin{document}
\title{The mixed problem for the Lam\'e system in two dimensions}

\author{
K.A.~Ott\footnote{Katharine Ott is partially supported by a grant from the
  U.S.~National Science Foundation, 
  DMS 1201104.} 
\\Department of Mathematics\\University of Kentucky\\Lexington,
KY 40506-0027, USA
\and
 R.M.~Brown\footnote{
Russell Brown   is  partially supported by a grant from the Simons
Foundation (\#195075).
} \\ Department of Mathematics\\ University of Kentucky \\
Lexington, KY 40506-0027, USA}
\date{}
\maketitle

\abstract{ 
We consider the mixed problem for $L$ the Lam\'e system of
elasticity in a bounded Lipschitz domain $ \Omega\subset\reals
^2$. We suppose that the boundary is written as the union of two
disjoint sets, $\partial\Omega =D\cup N$. We take traction data
from the space $L^p(N)$ and Dirichlet data from a Sobolev space $
W^{1,p}(D)$ and look for a solution $u$ of $Lu =0$ with the given
boundary conditions. We give a scale invariant
condition on $D$ and find an exponent $ p _0 >1$ so that for $1<p<p_0$,
we have a  unique solution of this boundary value  problem  with the
non-tangential maximal function of the gradient of the solution  in 
$L^ p(\partial\Omega)$. We also  establish the existence of
a unique solution when the data is taken from Hardy spaces and
Hardy-Sobolev spaces with $ p$ in $(p_1,1]$ for some $p_1 <1$. 
  }

\section{Introduction}\label{Introduction}
We consider the $L^p$-mixed problem for the 
Lam\'e system of elasticity in  a Lipschitz domain $ \Omega \subset
\reals ^2$. Thus we consider the operator $L= \mu \Delta + (\lambda + \mu)
\nabla \div$ acting on vector-valued functions. We assume that the  Lam\'e
parameters satisfy $ \mu >0$
and $\lambda>  -\mu$   so that the operator $L$ is elliptic. 
(See the discussion after  (\ref{KornElliptic}) for more details.) 
We assume that we have a decomposition of the boundary
into two sets $ \partial \Omega = D\cup N$ where $ D \subset \partial
\Omega$ is a non-empty, open, proper subset of $ \partial \Omega$ and
$N = \partial \Omega \setminus D$. We let $ \partial u / \partial
\rho$ denote a  traction operator at the boundary. Given data $ f_N$
on $N$ and $ f_D$ on $D$, we look for a solution $u$ of the boundary
value problem
\begin{equation} 
\label{MP}
\left \{ \begin{array}{ll} L u = 0, \qquad &\mbox{in }\Omega\\
u = f_D, \qquad  & \mbox{on } D\\
\frac { \partial u }{\partial \rho } = f _N , \qquad & \mbox{on } N\\
( \nabla u ) ^ * \in L^ p ( \partial \Omega) . \qquad & 
\end{array}
\right. 
\end{equation}
Here, $ \nontan{( \nabla u )} $ denotes the non-tangential maximal
function of $\nabla u$. 
The goal of this paper  is to give conditions on $ D$ and the data $f_D$
and $f_N$  that  will allow us to establish the existence and
uniqueness of the solution $u$ and show that the solution $u$ depends
continuously on the data. 

This investigation is a continuation of work that dates back at least
to Dahlberg \cite{BD:1977} who studied the Dirichlet problem for the
Laplacian in a Lipschitz domain by a careful investigation of harmonic
measure. Further developments for the Laplacian include 
Jerison and Kenig \cite{JK:1982c} who treated the Neumann problem and
the regularity problem in $L^2 ( \partial \Omega)$ and Dahlberg and
Kenig \cite{DK:1987} who studied the Neumann and regularity problems
in $L^ p ( \partial \Omega)$ for $p$ between 1 and 2. We point out
that it has become standard to refer to the Dirichlet problem with
data that has one derivative in some $L^p( \partial \Omega)$ space as
the regularity problem. In both the regularity and the Neumann
problem, the goal is to establish estimates on the non-tangential
maximal function of the gradient of the solution. Other relevant
developments include the work of Dahlberg, Kenig, and Verchota
\cite{DKV:1988} who studied the Dirichlet, Neumann, and regularity
problems for the Lam\'e system in $L^2( \partial \Omega)$ and work of
Dahlberg and Kenig \cite{DK:1990} who studied the regularity and
Neumann problems for the Lam\'e system in $L^p ( \partial \Omega)$ in
dimension 3.  Additional results were obtained for the Lam\'e system by
Mendez and  M.~Mitrea \cite{MR1781091} and  Mayboroda and M.~Mitrea
\cite{MR2181934} and again their work is 
limited to low dimensions.

Another strand of this story is the study of the mixed problem in
Lipschitz domains. The author Brown and collaborators including
Capogna, Lanzani, and Sykes \cite{RB:1994b,LCB:2008,JS:1999,SB:2001}
have established well-posedness of the mixed problem for the Laplacian
in restricted classes of Lipschitz domains.   
I.~Mitrea and M.~Mitrea
\cite{MM:2007} have extended these results to cover the Poisson
problem for the Laplacian in a wide range of function spaces. 
 These methods have been
extended to the Lam\'e system with I.~Mitrea \cite{MR2503013}, the
Stokes system with I.~Mitrea, M.~Mitrea, and Wright \cite{MR2563727},
and to the Hodge Laplacian by Gol'dshtein, I.~Mitrea, and M.~Mitrea
\cite{MR2839867}.  The aforementioned  works rely on the use of  the
Rellich identity with 
a vector field $\alpha$ chose that 
 $ \alpha \cdot \nu$ changes signs as we pass
from  $D$ to $N$. Here $ \nu$ is the outer unit normal to $ \partial
\Omega$.  It seems unlikely that this technique can be
applied in a general Lipschitz domain. In two recent papers, the
authors and Taylor \cite{OB:2009,TOB:2011} 
have developed techniques to investigate the mixed problem for the
Laplacian in a general Lipschitz domain and for quite general
decompositions of the boundary. Another interesting approach was taken
by Venouziou and Verchota \cite{MR2500502} who study the mixed problem
for the Laplacian in a class of polyhedral domains. The
Ph.D. dissertation of Venouziou \cite{MR2873458} includes a discussion
of mixed problems for the biharmonic equation in a class of polyhedral
domains and some examples where these mixed problems are not solvable.
Venouziou also gives a nice summary of early work on the mixed
problem.  

The work reported here extends the methods of Taylor, Ott, and Brown
to the Lam\'e system in two dimensions. A key ingredient of our
approach to the mixed problem (\ref{MP}) is estimates for the Green
function for the mixed problem.  Taylor, Ott, and Brown
\cite{JT:2011,TOB:2011} obtain the needed estimates for the Green
function for the Laplacian in all dimensions by the method of de
Giorgi \cite{EG:1957}.  (As we are considering constant coefficient
operators, solutions are well-behaved in the interior of the domain.
The interesting issue is regularity at the boundary.)  However, this
approach is not available for elliptic systems.  The Green function
for the mixed problem in two dimensions was studied recently by
Taylor, Kim, and Brown \cite{TKB:2012}.  We note that the work of
Dahlberg and Kenig \cite{DK:1990} provides estimates for 
the Green function for the Dirichlet problem and the Neumann problem
in three dimensions.  However, this approach depends on estimates for
the $L^p$-Dirichlet problem for $p$ near 2.  It is not clear how to
extend this approach to mixed boundary value problems. The
investigation of the mixed problem for elliptic systems in higher
dimensions remains an interesting open question.

With the estimates for the Green function in hand, we argue as in the
work of Dahlberg and Kenig \cite{DK:1987} and use the asymptotic
behavior of the Green function to study our boundary value problem for
$p\leq 1$.  This gives existence of solutions when the data comes from
atomic Hardy spaces.  Finally, we adapt an argument of Shen
\cite{ZS:2007} to establish results in $L^p$ for $p>1$ and
obtain the  following theorem which is the main result of this
paper. 
\begin{theorem} 
\label{Main}
Let $ \Omega \subset \reals ^ 2$ be a Lipschitz domain and let $ D
\subset\partial \Omega$ be a non-empty open proper subset satisfying
the corkscrew condition (\ref{Corkscrew}) and the conditions
(\ref{BIntegral}) and (\ref{Integral}).  Let $L$ be the Lam\'e
operator with coefficients satisfying  (\ref{Elliptic}). 

1) There exists $p_0 >1$ so that for $ 1 < p < p _0$, the $L^ p 
$-mixed problem is well-posed in the sense that if $
f_N \in L^ p (N)$, $ f_D \in \sobolev p 1 ( \partial \Omega)$, the
problem 
(\ref{MP}) has a unique solution which satisfies the estimate
$$
\| ( \nabla u ) ^ * \|_ { L^ p ( \partial \Omega) } 
\leq C ( \|f _N \| _ { L^ p ( N) }+ \| f_D \| _ { \sobolev p 1 (
  \partial \Omega) } ) .
$$
The boundary values of $u$ and $\nabla u$ exist as non-tangential
limits. 

2) There exists $p_1 < 1$ so that if $ p _ 1< p \leq 1$, the
$L^p$-mixed problem is well-posed. Thus if   $f_N \in H^ p (N)$ 
 and   $f_D$ lies in the Hardy-Sobolev space 
$H^ { 1,  p}( \partial \Omega  )$, then there exists a 
unique solution of the $L^p$-mixed problem which satisfies 
$$
\| u  \| _ { H ^ {1,p} ( \partial \Omega)} +
\| \frac {\partial u } { \partial \rho } \| _ { H ^ p ( \partial \Omega)}
+
\| \nontan { ( \nabla u )}\| _ { L^ p ( \partial \Omega)}
\leq C ( \| f _N \| _ { H^ p (N) } + \| f_D\| _{ H^ {1, p } ( \partial \Omega
) } ) .
$$
\end{theorem}

We point out that we 
will  require that our Dirichlet data have an extension
from $D$ 
to $ \partial \Omega$.  The sets $D$ that we consider may not be
extension domains for Sobolev spaces.  As our solutions (at least for
$p\geq 1$) will have boundary values in $ \sobolev p 1 ( \partial
\Omega)$, it is clear that we need to assume that $f_D$ is the
restriction to $D$ of a function which lies in $\sobolev p 1 (
\partial \Omega)$.

Section \ref{Definitions} gives the definitions and assumptions needed
in the rest of the paper. Section \ref{Reverse} gives a reverse
H\"older estimate for the gradient of a weak solution of the mixed
problem.  In section \ref{Atoms}, we give the fundamental estimate for
solutions with atomic data. Uniqueness is treated in section \ref{JustOne} and
the proof of the uniqueness assertion of Theorem \ref{Main} is given
in Theorem \ref{Uniq}.  The existence of solutions for the Hardy space
problem with $ p\leq 1$ is given in Theorem \ref{pEx} of section
\ref{Exist} and the existence for $p>1$ is indicated in Theorem
\ref{LpEx} which completes the proof of Theorem 1.2. 
  Two appendices summarize information that is probably
known, but not available in a convenient format. Appendix
\ref{Inequalities}  gives  several versions of the Sobolev
inequalities that are essential to this work and Appendix
\ref{RegularityProblem} provides a treatment of the regularity problem
with data from the Hardy-Sobolev space $H^ { 1, p }( \partial
\Omega)$.

\section{Definitions}
\label{Definitions}

\subsection{Domains}
\label{Domains}

We will assume that $ \Omega $ is a bounded Lipschitz domain. While
our final results hold only in two dimensions, many steps of the
argument can be carried out in any dimension.  When possible we will
give our arguments in $ \reals ^n$, $ n \geq 2$, in order to highlight 
that part of the argument that is truly two-dimensional and to lay out
the issues that restrict our argument to two dimensions. 
We assume $\Omega$ is a bounded open set in $\reals ^n$ and for $M>0$,
$ r>0$, and $ x \in \partial \Omega$, we define $Z_r(x) $  by 
$
Z_r(x) = \{ y : |x' - y ' | < r , |x_n - y _n | < ( 4M+2) r \}$. We
say that $ Z_r(x) $ is a  {\em coordinate cylinder for $ \partial
  \Omega$ }if there is a Lipschitz function $ \phi : \reals ^ { n -1 }
\rightarrow
\reals $ with Lipschitz constant $M$ so that 
\begin{eqnarray*}
\partial \Omega  \cap \cyl x r  & =  & \{ y : y _ n = \phi (y ') \}
\cap \cyl x r \\
\Omega \cap \cyl x r & = & \{ y : y _ n > \phi ( y ') \} \cap \cyl x r. 
\end{eqnarray*}
The coordinate system $ ( y ', y _n ) \in \reals ^ { n -1 } \times
\reals$ is assumed to be a rotation and translation of the standard
coordinate system on $ \reals ^ n$. 
\note{ An argument that 4M+2 gives a star-shaped domain. 

If we let the star be centered at $4M+1$, then the min slope of a line
segment joining the square of sidelength 1/2 to a point on the base is
r( 4M+ 1/2 - Mr) / ( 3r/2) = (3M + 1/3) . And since this is strictly
greater than M we are ok. (Though we may replace 4M+2 by a smaller
value. 
} 

We say that $ \Omega $ is a { \em  Lipschitz domain  }with constant $M$ if for
each $ x \in \partial \Omega$, we may find a coordinate cylinder for $
\partial \Omega$
centered at $x$, $Z_r(x)$. Since $ \partial \Omega$ is compact, we may find $
r_0$ and a finite collection of coordinate cylinders $ \{ \cyl   {x_i
}  { r     _0 }   \} _ {i = 1} ^ N$  which cover $\partial \Omega$  and so
  that for each $i$, $ \cyl {x_ i } { 101 r_0 }$ is also a coordinate
cylinder. 

Our estimates near the boundary will be defined using the
non-tangential maximal function. 
To define this object, we  fix $
\alpha > 0$ and for $ x\in \partial \Omega$, we 
define the  {\em non-tangential approach region with vertex at $x$
}by $\Gamma (x) = \{ y \in \Omega : |x-y | < (1+ \alpha ) d(y) \}$
where $ d(y ) = \dist (y, \partial \Omega)$ denotes the distance from
$y$ to the boundary.  For a function $v$ defined in $ \Omega$, we
define the {\em  non-tangential maximal function }of $v$ by 
$$
\nontan {v } ( x) = \sup _ {y \in \Gamma (x) } |v( y ) |, \qquad x \in
\partial \Omega.
$$
It is well-known that the $ L^ p$ norm with respect to surface measure
of the non-tangential maximal functions defined using different
choices of  $ \alpha$ are comparable.  Thus, we suppress the
dependence of $\Gamma$ and $ \nontan v $ on the parameter $ \alpha$. 

We recall a useful tool from the Ph.D. dissertation of G.
Verchota \cite[Theorem A.1]{GV:1982}. Given a Lipschitz domain  $
\Omega \subset \reals ^ n$, we may find a family of 
domains $ \{ \Omega _k \}_{ k=1} ^ \infty $ with $ \bar \Omega _ k \subset \Omega$
with the properties that a) each $ \Omega _k $ is a $C^
\infty$-domain, b) there are bi-Lipschitz transformations $ \Lambda _ k
:  \partial \Omega \rightarrow \partial \Omega _k $ such that $ \Lambda
_k (x)$ converges to $ x$. Using Verchota's construction we can see
that the $ \Lambda _k (x) $ converges to $x$ non-tangentially as $k
\rightarrow \infty$, 
 i.e. if the constant  $ \alpha $ used to define the approach regions
 $ \Gamma (x)$ is sufficiently
large, then $ \Lambda _k ( x) \in \Gamma (x)$ for all $k$. 

We define  a  star-shaped Lipschitz domain. These domains
will be useful because they have sufficiently regular geometry that we
can establish  estimates that are uniform over the  class of domains. 
Let $ x \in \reals ^ n$, $ r>0$, $M > 0$, and $ \phi : \sphere ^ { n
  -1} \rightarrow [ 1, 1+M]$. We say that $ \Omega$ is a {\em
  star-shaped Lipschitz domain } with center $x$, constant $M$, and
scale $r$ if  $\phi$ is Lipschitz with constant $M$ and 
$$
\Omega = x + \{ y : |y | < r \phi( y/ |y|)\}. 
$$

We introduce {\em boundary balls }(or intervals if $n =2$). For $ x
\in \partial \Omega$ and $ r \in (0, 100r_0)$, we let $ \sball x r =
\cyl x r \cap \partial \Omega$.  If $ x\in \Omega$ and $ r \in (0,
100r_0)$, we define {\em local domains } as follows. If $ d(x) =
\dist(x, \partial \Omega) > r$, then $ \locdom x r = \ball x r$, the
ball centered at $x$ with  radius $r$. If $ d(x) \leq r$, then 
 $x$ lies in one of the coordinate cylinders $\cyl {x_i} { r _0 } =Z_i
$. If $ x = (x', x_n)$ in the coordinate system for the cylinder
$Z_i$, then we put $ \hat x = ( x',\phi (x'))$ and define $ \locdom x
r = \cyl {\hat x } r $. 
We let $ x^ * = (x', \phi(x') + (4M+1) r e_n )$
and  observe that the domains $ \locdom x r$ are  star-shaped with
respect to each point in $ \ball {x^*}{ r/2}  $, thanks to our choice
of $M$. This 
allows us to prove scale-invariant Sobolev inequalities and Korn
inequalities in these domains.  If $x$ is on or near the boundary,
then $x$ may lie in several of the coordinate cylinders which cover
the boundary and thus we have several choices for $ \sball x r $ and $
\locdom x r$.  Our estimates will hold for any such choice with the
condition that if several of these objects appear in an estimate we
make a consistent choice of coordinate cylinder to define them.  Note
that in earlier works, we have sometimes used $ \ssb x r = \partial
\Omega \cap \ball x r$ and $ \ddb x r = \Omega \cap \ball x r$, $x \in
\Omega$, in
place of $ \sball xr $ and $ \locdom x r$. Each set of definitions has
its advantages and at one point in the argument below we will
find it convenient to use $ \ssb x r$ in place of $ \sball x r$.

We specify our decomposition of the boundary by giving our
assumptions on $ D$. We assume $ D\subset \partial 
\Omega$ is a non-empty proper  open subset of  $ \partial
\Omega$. We will assume that $D$ satisfies a corkscrew condition. To
describe this condition, let $\Lambda \subset \partial \Omega$ denote
the boundary of $D$ in $ \partial \Omega$ and $ \delta(x) = \dist(x, \Lambda
)$. We say that $D$ satisfies the corkscrew condition
if for all $x \in \Lambda$ and $ r $ in $(0, 100r_0)$, there exists $
x_r \in D$ so that $ |x_r -x | < r$ and $ \delta (x _r) > M ^ { -1}
r$. In Taylor, Ott, and Brown \cite[Section 2]{TOB:2011}, we show that
if  $D$ satisfies the  corkscrew condition, we have the following
consequence:
\begin{equation}\label{Corkscrew} 
\parbox{4 in} {If $ x \in D$ and $ r \in (0,100r_0)$, then there
  exists $ x_r \in
  D$ so that $ |x-x_r| < r $  and $ \sball {x_r} { M^ { -1} r } \subset D$
  .} 
\end{equation}

We note that the extreme cases $ D = \emptyset$ or $ \partial \Omega$
of the mixed problem (\ref{MP}) correspond to the traction and
regularity problems, respectively. As these have been studied
elsewhere, we exclude these cases from our discussion in this paper. 

We also require that $D$ satisfy the following conditions that involve
the integrability of $ \delta$. To state these conditions, for $  r
\in 
(0, 100 r _0)$, we let $ \delta _r (x) = \min ( \delta (x) , r)$ and
then we require that 
\begin{eqnarray}
\label{BIntegral} 
\int _ { \sball x r } \delta_r ^ t \, d\sigma \approx r ^ { n -1 + t }  ,
\qquad t  > -1 + \epsilon \\
\label{Integral} 
\int _ { \locdom x r } \delta_r ^ t \, dy  \approx r ^ { n  + t  }  ,
\qquad t > -2 + \epsilon . 
\end{eqnarray}
In our main results there will be a restriction on $ \epsilon$. The
condition on $ \epsilon$ depends on the $L^q$-index in the reverse
H\"older inequality in Theorem \ref{RH} and we refer
the reader to section \ref{Atoms} for more details. For the moment, we
observe that given $M$ we will find a 
positive value of $ \epsilon$ for which the conditions
(\ref{BIntegral}) and (\ref{Integral}) allow us to solve the mixed
problem.  While the conditions (\ref{BIntegral}) (\ref{Integral}) may
seem rather mysterious, they are closely related to the dimension of
the set $ \Lambda$. The work of Taylor, Ott, and Brown \cite[Lemmata
2.4, 2.5]{TOB:2011} shows that the conditions (\ref{BIntegral}) and
(\ref{Integral}) follow if we assume that the set $ \Lambda$ is of
dimension $ n- 2 + \epsilon$. In the setting treated in
\cite{OB:2009}, we have $ \epsilon =0$. As the conditions
(\ref{BIntegral}) and (\ref{Integral}) become more restrictive as $
\epsilon$ decreases, the results of this paper always apply in the
domains considered in \cite{OB:2009}.

{\em Example. } We close this sub-section by giving an example of a
domain that satisfies our conditions and illustrates that even in two
dimensions, the set $D$ can be fairly complicated.  Thus let $ \Omega
= \{ x : |x| < 1 \}$ be the disk and let $D$ be the set $ \cup _ { k
=1 } ^ \infty \{ ( \cos \theta , \sin \theta ) : \pi / { 2 ^ {2 k +
1} } < \theta < \pi / 2 ^ {2k} \}$.  It is not difficult to see
that this domain will satisfy (\ref{Corkscrew}) and (\ref{BIntegral})
and (\ref{Integral}) with $ \epsilon =0$.

\subsection{Function spaces} 
\label{Functions}

We will consider $L^p$ spaces with respect to Lebesgue measure on domains
in $ \Omega\subset  \reals^n$ and with respect to surface measure, $\sigma$, on 
the boundary of $\Omega$, $\partial \Omega$. These will be denoted $L^p ( \Omega)$ and
$L^p( \partial \Omega)$, respectively. 

We let $ \sobolev p 1 ( \Omega)$ denote the standard  Sobolev space of
functions having one  derivative in $L^p ( \Omega)$ with the norm
$$
\| u \| _ { \sobolev p 1 ( \Omega) } = \left ( \int _ \Omega |\nabla u
|^ p
   + r_0 ^ { -p } |u |^ p \, dy \right) ^ { 1/p }
$$
at least for $ 1\leq p < \infty$.  The factor of  $ r_0$ guarantees
that the two terms in the norm have the same homogeneity when
rescaling.  Our functions will generally take
values in $ \reals^n$, though we will not indicate the range in our
notation.

For $ D \subset \partial \Omega$, we define $ \sobolev 2 1 _D ( 
\Omega)$ as the closure in $ \sobolev 2 1 ( \Omega)$ of the functions
which are smooth in the closure of $ \Omega$ and which vanish in a
neighborhood of $D$. We let $\sobolev 2 { -1} _D ( \Omega)$ denote the
dual of $ \sobolev 2 1 _D ( \Omega)$. As noted all domains in this
paper will be Lipschitz and thus we have the trace operator which is a
continuous map from 
 $ \sobolev 2 1 ( \Omega)$  into $ L^ 2 ( \partial  \Omega)$ and
 extends the operation of restricting a smooth function to the
boundary. We let $ \sobolev 2 { 1/2} _D ( \partial \Omega)$ denote the
image of $ \sobolev  2 1 _ D(\Omega)$ under the trace map and then $
\sobolev   2 { -1/2} _D ( \partial \Omega)$ denotes the dual of $ \sobolev
2 { 1/2} _D ( \partial \Omega)$. 

We will also need to consider Sobolev spaces on the boundary. For a function
$u$ in $ C^ \infty ( \bar \Omega)$  we  define  a family of  tangential
derivatives 
$$
\frac  {\partial u }{ \partial \tau _{ij} }  = \nu _i \frac { \partial u }{ \partial x_j} - \nu _ j
\frac { \partial u  }{ \partial x _ i }, \qquad i,j=1,\dots,n, i \neq j 
$$
where $ \nu$ denotes the unit outer normal to the boundary. 
For a function $u$ which is smooth in $ \bar \Omega$, we
 can define the tangential gradient by 
$$
\nabla _t u^ \alpha = \nabla u^ \alpha -\nu \nabla u^ \alpha \cdot
\nu, \qquad \alpha = 1, \dots,n. 
$$
Note that we have 
$$
|\nabla _ t u|^2  \approx  \sum  _{1 \leq  i < j \leq n } 
|\frac { \partial u }{ \partial \tau _ { ij}}|^2
$$
and  for $1\leq p < \infty $,  the Sobolev space  $ \sobolev p 1 (
\partial \Omega)$  is the closure of the
boundary values of smooth functions in the norm 
$$
\| u \| _ { \sobolev  p 1 ( \partial \Omega ) }  = (  \|\nabla _t u
\|^ p _ { L^ p (\partial  \Omega) }  
+ r_ 0 ^ {-p}\| u \|^p _ { L^ p ( \partial \Omega)}) ^ { 1/p}. 
$$ 

Before defining Hardy spaces on the boundary, we recall the definition
of the spaces of H\"older continuous functions.  For $K $ a compact
subset of $ \bar \Omega$ and $ 0 < \alpha \leq 1$, we define the
H\"older space $ C ^ \alpha ( K ) $ to be the collection of functions
$f$ on $K$ for which the norm below is finite
$$
\| f \| _ { C ^ \alpha ( K) } = \sup _{ x \neq y } \frac {  |f(x)  -
f ( y) | } { |x-y | ^ \alpha }  + r _ 0 ^ { - \alpha } \sup _ K |f(x) | .
$$

If $ \Omega $ is a Lipschitz domain in dimension $n$ and $N \subset
\partial \Omega$ is a closed  set, we introduce the Hardy space $ H^p(N)$ for $ (n-1)/n< p \leq 1$. We
say that $ a $ is an atom for $ \partial \Omega$ if $ a $ is supported
in a boundary ball $ \sball x r$, $ a $ satisfies $ \| a\|_ { L^ \infty
  ( \partial \Omega)} \leq 1/\sigma ( \sball x r)$ and $ \int _ {
  \sball x r } a \, d\sigma =0$.  
We say that $a$ is  an atom for $N$ if $ a =\tilde a |_N$ for some 
 $\tilde a$, an atom for $\partial \Omega$. As we shall see these
atoms are building blocks for the spaces $H^ p (N)$ which arise
naturally in the study of $L^p$-mixed problem for $ p \leq 1$.  We
note that our atoms are normalized for the space $H^1$. When we
introduce $H^p$ spaces with $p < 1$, we will need to introduce powers of
 $\sigma( \sball x r) $  into the norm to compensate for the fact that the atoms
are not normalized in $H^p$. 

For $ p $ satisfying  $(n-1)/n < p \leq 1$, we say that $f$ lies in
the {\em Hardy space }$H^p(N)$ 
if there exists a sequence of atoms for $N$, $\{ a _ j \}_ { j =1 } ^
\infty  $ and
coefficients 
$\{ \lambda _j \} _ { j =1 } ^ \infty \subset (0, \infty )$ so that
$f = \sum _ { j = 1 } ^ \infty \lambda _j a _j $ and 
$ \sum _ { j =1 } ^ \infty \lambda _ j ^ p \sigma (\sball { x_j } {
  r_j})  ^ {  1-p  }
< \infty $ where the atom $ a _j$ is supported in a boundary
ball  $\sball { x_j}{r _j}$. We  define a quasi-norm for
$ p \leq 1$ by  
$$
\| f \| _ { H ^ p (N) } = \inf \left ( \sum _ j \lambda _ j ^ p \sigma (
\sball  { x _ j }{ r _  j} )  ^ { 1- p  } \right ) ^ { 1/p}
$$ 
where the infimum is taken over all representations of $f$ as a sum
of atoms. When $p =1$, the sum converges in $L^1 $ and we have $ H ^ 1
( N) \subset L^ 1 (N)$. However, for $p < 1 $ elements of $H^ p ( N)$
may not be functions. Rather they are defined as linear functionals on
a space of smooth functions. It is well-known that elements of $ H ^
p(\partial \Omega)$ give continuous linear functionals on the space $
C^ \alpha ( \partial \Omega)$ where $ \alpha = ( n-1 ) ( 1/ p - 1 ) $,
for $ 1>p > ( n-1) /n$. It is a straightforward consequence of our
definition that elements of $ H ^ p ( N)$ give continuous linear
functionals on $ C^ \alpha _D( \partial \Omega)$, the collection of
functions in $ C^ \alpha ( \partial \Omega) $ which vanish on $D$.  We
observe that it is an immediate consequence of the definition that an
element $f$ from $ H^p(N)$ has an extension to $ H^ p ( \partial
\Omega)$ with the bound $ \| f \| _ { H^ p (\partial \Omega)} \leq C
\| f \| _{H^ p (N)}$ for any $C>1$.  One consequence of our  main
theorem  is that there is a bounded linear extension operator from $
H^p (N)$ to $ H^ p ( \partial \Omega)$. However, it would be
interesting to develop a better understanding of these spaces.

We  define Hardy-Sobolev spaces $H^ { 1, p } ( D) $ as
follows.  We say that $A$ is a 1-atom for $ \partial \Omega$ if $ A$
is supported in a boundary ball $ \sball x r$ and $\| \nabla _t A \| _
{ L^ \infty ( \partial \Omega ) } \leq 1 / \sigma ( \sball x r)$.  We
  say
that $A$ is an atom for $D$, if $A$ is the restriction to $D$ of an
atom for $ \partial \Omega$. Finally, if $p$ satisfies $ (n-1)/n < p
\leq 1$,  we say $u$ is in $ H^ { 1, p
} ( D)$ if there is a sequence of atoms $\{ A_j \}$ for $D$ with each
$ A_j $ supported in a boundary ball $ \Delta _j = \sball { x_j} { r_j }$  and a sequence of
non-negative real numbers  $
\{ \lambda_j \} _ { j =1 }^ \infty$ with $ \sum _ j \lambda _j \sigma (
 \Delta _j ) ^ { 1 - p } < \infty $  and   $ u= \sum _ j
\lambda _j A_j $.  We define a quasi-norm on $ H ^ { 1, p } ( \partial
\Omega)$ as  $\| u \| _ { H ^ { 1, p } ( \partial
  \Omega) } = \inf(  \sum _ j  \lambda _ j ^ p \sigma ( \Delta _ j ) ^ {
    1-p } ) ^ { 1/ p }$ where the infinum is taken over all possible
  representations of $u$.   
It is well-known that one may define atoms 
using an $L^t$ space  $ 1< t< \infty$ instead of $L^ \infty$.
Thus, if we fix $t$ with $ 1 < t <
\infty$ and  replace the condition $ \| a\| _ { L^ \infty ( \sball x r
  ) } \leq \sigma (\sball x r )^ { -1}$ by $ \| a \| _ { L^ t ( \sball xr   )}
\leq \sigma ( \sball x r) ^ { -1/ t'}$, 
then the resulting Hardy space is independent of $t$.   A similar
result holds for the Hardy-Sobolev space. 
The work of Mitrea and
Wright  \cite[sections 2.2-2.3]{MW:2011}   provides an exposition
of these facts in the setting of the 
 boundary of a Lipschitz domain. 

\note{ Sum converges where?  For each $j$ $\| A_j \| _ 1 \leq r_j $,
  thus $\| \sum \lambda _j A_j \| _1 \leq \sum \lambda _j r_j \leq 
(\sum \lambda _j  \sigma ( \Delta _ j ) ^ { \frac 1 p -1})
\sup _j r _j \sigma ( \Delta _j ) ^ { \frac 1p -1 }  
\leq( \sum \lambda _j ^ p \sigma (\Delta _j ) ^ {\frac 1 p -1 } ) ) ^
    { 1/p} \sup r_ j ^ { n - ( n-1) /p} $ .
Since $ \partial \Omega$ is bounded, the supremum will be finite if
$ n - ( n-1) /p \geq 0$ or if $ p \geq ( n-1)/n$. 

 } 

\note{ On what sets should we define $H^p(N)$?

Boundary  ball $ \Delta$, ???  ball $ S$ } 

\note{ 
 I don't think we really should be restricting a distribution to
  a closed set. Perhaps, we should really view these critters as
  defined on the interior of $N$. ????  

Also, we have that $ C^ \alpha$ is the dual of $H^p$. 

We seem to need more on the existence of solutions to the $L^p$-regularity
problem. 

A useful inequality 
$ ( \sum x_i ) ^ { p } \leq \sum x_ i ^ p$  if $ 0< p \leq 1$ and $
x_i > 0 $. 

 }

\subsection{The weak formulation of the mixed problem}

Next we recall the Lam\'e operator  $L$ and give several well-known
estimates for solutions of the Lam\'e system, $Lu=0$. The operator $L$ will  be
written as
$$ 
(Lu)^ \alpha = \frac \partial {\partial x_i } a^ { ij } _{ \alpha \beta} \frac
\partial {\partial x_j } u ^ \beta,\qquad \alpha, \beta, i, j =1,\dots,n.
$$
Here and throughout this paper, we will use the summation convention
and thus we sum on the repeated indices $i$, $j$ and $\beta$  in the
above expression.  
We will consider the traction operator at the boundary 
$$
\frac { \partial u  } { \partial \rho } ^ \alpha = \nu _ i a ^ { i j }
_{ \alpha \beta } \frac { \partial u ^ \beta } { \partial x_ j } . 
$$
The coefficients $ a^ { ij } _ { \alpha \beta} $ we consider are 
 given by 
\begin{equation}
\label{CoeffDef}
a^ { ij } _ { \alpha \beta } = a^ { ij } _ { \alpha \beta} (s) 
= \mu \delta _ { ij} \delta _ { \alpha \beta} + s \delta _ { i \beta}
\delta _ { j \alpha }  + ( \lambda + \mu - s ) \delta _ { i\alpha }
\delta _ { j \beta}
\end{equation}
where $\delta _{ ij}$ is the Kronecker symbol, $ \lambda $ and $\mu $ are the Lam\'e parameters and $s$ is an
artificial parameter that has been introduced in several different
settings for mathematical reasons. Note that 
$$
\frac \partial { \partial x_i } ( \delta _ { i \beta} \delta _ { j
  \alpha } - \delta _ { i \alpha }\delta _ { j \beta } ) \frac
{\partial} { \partial x_ j  } = 0 . 
$$
Thus, the differential operator $L$ does not depend on the value of
$s$. However, the traction  operator $ \partial  u / \partial \rho$ 
does depend on $s$. 

We will require that our operator satisfy the ellipticity condition 
\begin{equation}
\label{KornElliptic}
a^ { ij } _ { \alpha \beta } \xi ^ i _ \alpha \xi ^ j _ \beta \geq 
M^ { -1} |\frac { \xi + \xi ^ t } 2 | ^ 2, \qquad \xi \in \reals ^ { n
  \times n}.
\end{equation}
This will hold  if $ s$, $ \lambda$, and $ \mu$ satisfy 
\begin{equation}
\label{Elliptic} 
s\in [0, \mu ], \qquad \mu > 0,  \qquad \mbox{and} \qquad  2\mu + n \lambda > 0.
\end{equation}

We note  several consequences of this ellipticity assumption (\ref{Elliptic})
 which
come under the name  of Korn inequalities. 
First we  have the global coercivity estimate  
\begin{equation} 
\label{Global} 
c\int _ \Omega |\nabla u |^ 2 \, dy  \leq  \int _ \Omega
a ^ { ij } _ { \alpha \beta } \frac { \partial u ^ \beta} { \partial
  y_ j } \frac  { \partial u ^ \alpha } { \partial y _ i } \, dy , 
\qquad u \in \sobolev 2 1 _ D ( \Omega) . 
\end{equation}
This  will hold if $D$ is a non-empty open subset of the boundary. 
The constant will depend on $ \Omega$ and $D$ and the constant will be
bounded away from zero for a family of operators which satisfy
(\ref{KornElliptic}) for a fixed $M$. 

\note{ As far as I can tell this constant depends on $ \Omega$. 
}
For   $u \in \sobolev 2 1 ( \Omega)$,   any local domain $ \locdom x
\rho$,      and any constant vector    $ c \in \reals ^ n $, we have
\begin{equation}
\label{LocalElliptic} 
\int _  { \locdom x \rho } |\nabla u | ^ 2 \, dy 
\leq C (M) \int _ {\locdom x \rho }  a^ { ij }_ { \alpha \beta } \frac {
    \partial u ^ \beta } { \partial y_ j } 
\frac { \partial u ^ \alpha } {\partial y _ i }  + \rho ^ { -2} |u-c
|^ 2 \, dy. 
\end{equation}
The constant depends only on $M$ and  the estimate does not depend on
the properties of $D$ and, in fact,  we do not  require that $ u$
vanish on $D$ for the estimate (\ref{LocalElliptic}) to hold. 
A proof of (\ref{LocalElliptic}) may be found in the monograph of 
Ole\u\i nik, Shamaev, and Yosifian \cite[Theorem 2.10]{MR1195131}. 
Throughout this paper, we will consider the Lam\'e operator (and its
generalizations)  with coefficients given by (\ref{CoeffDef}) and
satisfying (\ref{Elliptic}). As a consequence of these assumptions,
we obtain the estimates (\ref{KornElliptic}), (\ref{Global}), and
(\ref{LocalElliptic}).

We are ready to  give the definition of a weak solution
of the mixed problem. Given $ f  \in \sobolev 2 {-1}  _ D( \Omega)$ and $
f_N \in \sobolev 2 { -1/2} ( \partial \Omega)$, we consider the 
boundary value problem
\begin{equation} 
\label{WMP} 
\left \{ \begin{array}{ll}
Lu = f , \qquad & \mbox{in } \Omega\\
u = 0 , \qquad & \mbox{on } D \\
\frac { \partial u } { \partial \rho } = f _N ,  & \mbox {on } N . 
\end{array}
\right. 
\end{equation}
We say that  $u$ is a  {\em weak solution of (\ref{WMP}) }if $ u \in
\sobolev 2 1 _ D ( \Omega)$ and 
$$
\int _ { \Omega } a ^ { ij } _ { \alpha \beta } \frac { \partial u ^
  \beta } { \partial y_ j } \frac { \partial \phi ^ \alpha } { \partial
  y_ i } \, dy = \langle f_N , \phi \rangle _ { \partial \Omega} -
\langle f , \phi \rangle _ { \Omega} , \qquad \phi \in \sobolev 2 1 _D
( \Omega) . 
$$
We are using $ \langle \cdot, \cdot \rangle_ \Omega$ and $ \langle \cdot , \cdot \rangle_{ \partial \Omega} $ to denote
the pairings of duality on $ \sobolev  2 {-1} _D( \Omega) \times \sobolev 2 {1} _D ( \Omega)$ and
$\sobolev 2 {- 1/2} _D ( \partial \Omega ) \times \sobolev 2 { 1/2}_D( \partial \Omega)$, respectively. 
As we assume that the coefficients of $L$ satisfy
(\ref{KornElliptic}) and we assume $D$ is a nonempty open set of
$\partial \Omega$, we have the coercivity estimate (\ref{Global}) and
a Poincar\'e inequality (\ref{Poincare}). 
If $ u \in \sobolev p 1 _D ( \Omega)$, then we have that 
\begin{equation}
\label{Poincare} 
\| u \| _ { L^ p ( \Omega) } \leq C r_0 \| \nabla u \|_ { L^ p (
  \Omega)}.
\end{equation}
Thus the existence and uniqueness
of solutions
to (\ref{WMP}) is a straightforward consequence of   standard Hilbert
space theory. The solution of (\ref{WMP}) will satisfy 
$$
\| \nabla u \| _ { \sobolev 2 1 _D  ( \Omega) } 
\leq C (  \|f \| _ { \sobolev 2 { -1} ( \Omega) } + \| f _N \| _ {
  \sobolev 2 { -1/2} _D (\partial \Omega) } ) .
$$

Now that we have introduced the traction operator, $ \partial  
/ \partial \rho$, we describe the sense in which we take boundary
values of $ \partial u / \partial \rho $ when  $p < 1$. 
 As the boundary values   may not be a function,
it is not enough to ask for non-tangential limits a.e. Thus for $p <
1$, and $ f \in H ^ p (\partial \Omega)$, we say that $\partial u / \partial \rho
= f $ on $ \partial \Omega$,  if  we have 
\begin{equation}
\label{HPBV}
\lim _ { k \rightarrow \infty } \int _ { \partial \Omega_k }  \phi^
\beta ( \frac { \partial u } { \partial
  \rho } ) ^ \beta \, d\sigma 
=  \langle f , \phi \rangle _ { \partial \Omega
}, 
\qquad \phi \in C^ \alpha  ( \bar  \Omega)
\end{equation}
where $ \Omega_k \subset \Omega $ is a family of approximating domains
as constructed by Verchota.   If $ f_N$ lies in $H^ p (N)$, we say that 
$\partial u / \partial \rho =f_N $ on $N$ if we have (\ref{HPBV}) for $
\phi$ which lie in $C^ \alpha _D(\partial \Omega)$. 

\note{ The use of $ f$ and $f_N$ was intentional. $ f$ for the
  boundary values everywhere and $f_N$ for boundary values on a
  subset.

Let me know if you think this  is a really bad idea.
}


One technical point about our argument for the mixed problem is that
we do not have solutions in smooth domains as a starting point. A
common strategy in studying the Dirichlet and Neumann problems is to
approximate a Lipschitz domain by a sequence of smooth domains, prove a uniform
estimate in the approximating domains, and take a limit. It is an interesting question to find an
approximation of the mixed problems by problems which have smooth
solutions.

Finally, we make a note about the constants in our estimates. Many of
our estimates are of a local nature and hold on scales $r$ with $ 0<r<
r_0$ and with a  constant that depends only on
the parameter $M$  that appears in the definition of a Lipschitz
domain, the corkscrew condition, the constants in the estimates
(\ref{BIntegral}), and (\ref{Integral}), and any $L^p$ indices that
appear in the estimate.   We will say a constant depends on the global
character of $ \Omega$ if  the constant also depends on 
 the collection of coordinate
cylinders appearing in the definition of the Lipschitz domain and
the constant in the Korn inequality (\ref{Global}).  

\section{A reverse H\"older inequality for weak solutions}
\label{Reverse}

The first step of our argument is to show that if $u$ is 
a solution of (\ref{WMP})  with
nice data, then  $\nabla u $ in 
$L^ q ( \Omega) $ for some  $q>2$. We use the reverse H\"older
argument of Gehring  \cite{FG:1973} and Giaquinta and Modica
\cite{MR549962}. The argument is similar to  one  given in our joint
work with Taylor 
\cite{OB:2009,TOB:2011},  but a few changes are needed
 since the  operators involved may not be strongly elliptic. The next
 Theorem gives a precise formulation of our result. We note that the
 proof of this result only requires the condition (\ref{Corkscrew}) on
 $D$.    

\begin{theorem} \label{RH} 
Let  $ \Omega \subset \reals ^ n$ be a Lipschitz domain and
suppose that $D$ satisfies (\ref{Corkscrew}) and that $L$ is the
Lam\'e system with coefficients satisfying 
(\ref{Elliptic}). Suppose that   $u$ is a solution of (\ref{WMP}) with $ f=0$ and
$ f_N$ a function. 
We may find $ q_0 >2$ so that if 
 $ 2< q < q _ 0 $ and $ f_N \in L^ { q ( n-1) /n}( N)$, then for any
local domain $ \locdom x r $ we have 
\begin{equation} 
\label{RHI}
\begin{split}
&\left ( \average _ { \locdom x r } |\nabla u |^ q \, dy \right) ^ { 1
  / q } \\
&\qquad \qquad \leq C \left [ \average _ { \locdom x { 2r}} | \nabla u | \,
  dy + \left( \average _ {\partial \Omega \cap \partial \locdom x { 2r} } | f_N
  |^ { q ( n-1 ) /n }  \,d \sigma \right ) ^ { n / ( (n-1) q) } \right] .
\end{split}
\end{equation}
The value of $q_0$ depends only on $M$.   
The constant in this estimate depends only on $M$, the dimension $n$,
and $q$. 
\end{theorem}

In the above theorem and throughout this paper, we  use $\average _E f
 $ to denote the average of $f$ over a set $E$ and we adopt the
convention that the  average over the empty set is zero. 

\note{ $ \locdom x r $ is only defined for $x$ in $\Omega$, not $ \bar
  \Omega$. Fixed by removing the bar below. } 

We introduce the notation $ \avg u  x r $ which will be useful 
for the arguments of this section. For $ u$ a function
defined on $ \Omega$, $ x \in  \Omega$ and $ r \in ( 0, 100 r_0 )$,
we let $ \avg u x r = 0 $ if $ \delta(x) \leq r $ and $ \avg u x r =
\average _ \locdom x r  u (y) \, dy $ if $\delta (x) > r$. We observe that 
one useful feature of this definition is 
that if $ \eta $ is in $ C^ \infty ( \bar \Omega)$ 
and $ \eta =0 $ outside $ \locdom x r$,  and $ u \in \sobolev 21 _D (
\Omega)$, then  $ \eta ( u - 
\avg u x r ) $ lies in $ \sobolev 2 1 _D ( \Omega)$.

We  give several  formulations of the Sobolev inequalities that we
will use in the sequel. The proofs of the estimates  (\ref{SoPo2}) and
(\ref{BoPo}) are given in 
Appendix \ref{Inequalities}. These take advantage of the corkscrew
condition for $D$ (\ref{Corkscrew}) and the definition of $ \avg ux r
$. Let $p$ and $q$ satisfy  $ 1 \leq p <  n$ and 
 $ 1/q = 1/p - 1/n$. Then for $x \in \Omega$ and $ 0 < \rho < r <
100r_0$, we have  
\begin{equation} 
\label{SoPo2} 
\left ( \int _{ \locdom x \rho } | u - \avg u x \rho | ^ q \, dy
\right) ^ { 1/q} 
\leq \frac { C(M,p,n) r ^ { n-1} } { ( r- \rho ) ^ { n-1} } 
\left ( \int _ { \locdom x r } |\nabla u | ^ p \,dy \right) ^ { 1/p
}. 
\end{equation}
For $1 \leq p < n $, we define  $q$ by $ 1/q
= n/ ( p ( n-1) ) - 1 / ( n-1) $. We observe that for   all local
domains $ \locdom x
\rho$ and $ \locdom x r$ with $ 0 < \rho < r $, we have 
\begin{equation}
\label{BoPo} 
\left ( \int _ { \partial \locdom x \rho \cap \partial \Omega } 
|u - \avg u x \rho | ^ q \, d\sigma \right) ^ { 1/ q} 
\leq \frac { C (M, p, n) r ^ { n -1 } } { ( r- \rho ) ^ { n -1 } } 
\left ( \int _ {\locdom x r  } |\nabla u | ^ p \, dy  \right) ^ {
  1/p } .
\end{equation}

We turn to the proof of the reverse H\"older inequality in Theorem
\ref{RH}.  The proof will require the following auxiliary function 
$$
P_r f(x) = \sup _ { 0 < s < r }  \frac 1 { s ^ { n -1} } \int _ { \partial \Omega \cap \partial
  \locdom x s } |f | \, d\sigma. 
$$
When we apply $P_r$ to the traction  data $f_N$ we will assume that $
f_N$ is extended to be a function on $ \partial \Omega$ by setting $
f_N (x) =0 $  on $D$. As observed in \cite{OB:2009}, for $ 1 < p \leq
\infty$,
$q = pn/(n-1)$, we have 
\begin{equation}
\label{Poisson}
\| P_r f\| _ {L^ q ( \Omega)} \leq C \| f \| _ { L ^ p ( \partial
  \Omega)}.
\end{equation}
\note{ Compare to estimate of Dindos, Mitrea. Fractional integrals
  with measures.

Lots of changes to the proof below. Read carefully.}
\begin{proof}[Proof of Theorem \ref{RH}]  We let $u$ be a weak
  solution of (\ref{WMP}) with $ f=0$ and  $f_N$ a function.  We fix
  $ x \in \partial \Omega$, $ \rho $ and $ r$ with $ r/2 < \rho < r <
  r_0$ and set $ \epsilon =  (r- \rho) /2$. 
We claim that with $ p$ satisfying $ p > 2 (n-1) /n$, we have
\begin{multline}
\label{RHClaim}
\left ( \average _ { \locdom x { r/2} } |\nabla u |^ 2 \, dy \right )
^ {1/2}
\\
  \leq C \left [
\left ( \average _ { \locdom x r } |\nabla u | \, dy \right) 
+ \left ( \frac  1 { r ^ { n -1} } \int _ { \partial \locdom x r \cap
  \partial \Omega} | f _N | ^ p \, d \sigma \right ) ^ { 1/p } \right ]. 
\end{multline}
Observe that the last term is bounded by $ P_ {2r}  (|f_N|^ p ) ^ { 1/p
} (x) $. 
Theorem \ref{RH} will 
follow from (\ref{RHClaim}) using  the argument in Giaquinta
\cite[p.~122]{MG:1983},  and the estimate (\ref{Poisson}).

Thus, we turn to the proof of (\ref{RHClaim}). We let $ \eta$ be a
cutoff function which is one on $ \locdom x \rho$ and  zero outside $
\locdom x { \rho + \epsilon } $ and observe that $v= \eta ^ 2( u -
\avg  u x { \rho + \epsilon } ) \in \sobolev 2 1 _D ( \Omega)$. We let 
$ E = \int _ { \locdom x { \rho + \epsilon } } |u | |\nabla u | \eta
        |\nabla \eta| \, dy $. Using the product rule, the local
        ellipticity assumption (\ref{LocalElliptic}), and that $u$
        is a solution, we obtain 
\begin{equation}
\begin{split} 
\label{RHStep1}
\int _ { \locdom x { \rho + \epsilon } } \eta ^ 2 |\nabla u |^2 \, dy 
& =  \int _ { \locdom x { \rho + \epsilon }} |\nabla \eta ( u - \avg u
  x { \rho + \epsilon } ) | ^ 2 \, dy   +C E\\
& \leq \int _ \Omega a ^ { ij } _ { \alpha \beta }   \frac { \partial
    u ^ \beta } { \partial y _ j } \frac { \partial v ^ \alpha } {
    \partial y _ i } \, dy   +C E\\
& =  \int _ { \partial \Omega } f _N ^ \beta  v ^ \beta \, d\sigma
  + CE. 
\end{split}
\end{equation}
We apply H\"older's inequality, Young's inequality with $
\epsilon$'s,  and the boundary Poincar\'e inequality (\ref{BoPo}) and
obtain for any $\gamma \in \reals $ that 
\begin{equation}
\label{RHStep2}
\begin{split}
|\int _ { \partial \Omega } f _ N ^  \beta  v ^\beta  \,d \sigma |
\leq & r ^ \gamma \left ( \int _ { \partial \Omega \cap \partial \locdom
  x { \rho + \epsilon   } } |f_N |^ p \, d\sigma \right ) ^ { 2 / p } 
 \\ \qquad & + \frac {C r ^ { - \gamma } }{ ( 1- \rho/r)^ { 2n-2}} \left
 ( \int _ {   \locdom x { \rho + 2 \epsilon } } |\nabla u | ^ t \, dy
 \right ) ^ { 2/t} ,  
\end{split}
\end{equation}
where $ p > 1$ and $p$ and $t$ are related by 
$ 1/t = 1 /n + ( n-1)/(np')$.   We return to
(\ref{RHStep1}) and use 
(\ref{RHStep2}) to estimate the boundary term. To handle the error term
$E$, we use Cauchy's inequality  with $ \epsilon$'s to subtract  the
term  $ \nabla
u$ in $E$ from  the  left-hand side of (\ref{RHStep1}) and conclude that 
\begin{equation}
\label{RHStep3} 
\begin{split} 
\int _ { \locdom x { \rho + \epsilon } } \eta ^ 2 |\nabla u |^ 2 \, dy 
 \leq   
&
C ( 
r^ { -2} \int _ { \locdom x { \rho + \epsilon } } |u - \avg u x {
  \rho + \epsilon } | ^ 2 \, dy 
+ 
r ^ \gamma \left ( \int _ {
    \partial \Omega \cap \partial \locdom x { \rho + \epsilon } }
|f_N | ^ p \, d\sigma \right ) ^ { 2/p } \\
 & \qquad+ 
\frac { r ^ { - \gamma} } {  ( 1-\rho/r) ^ { 2n -2}  } \left( \int _ {
    \locdom x {\rho + 2 \epsilon }} |\nabla u | ^ t \, dy \right ) ^ { 2/t}  )
 .
\end{split}
\end{equation}
We use (\ref{SoPo2}) to estimate the first term in (\ref{RHStep3}),
choose $ \gamma$ so that $ n+\gamma = 2n /p$ and hence $ n -\gamma  = 
2(n-1) / p$, divide by $ r^ n$, and recall that $ 2 \epsilon = r -
\rho$ to obtain 
\begin{equation*} 
\begin{split}
&\average_ { \locdom x { \rho + \epsilon}} \eta^2 |\nabla u |^2 \,  dy \leq   C 
[  
\left ( \frac 1 { r ^ { n -1} } \int  _ {\partial  \locdom
  x r  \partial \Omega } |f_N |^ p \, d\sigma \right) ^ { 2/p }  
\\ &
\qquad \quad + \frac 1 {( 1 - (\rho/r))^ { 2n -2}} ( \left ( \average
_{ \locdom x {  r  }}  |\nabla u | ^ { 2n / ( n+2)} \, dy \right ) ^
{ (n+2) / n }  + \left ( \average _ { \locdom x { r } } |\nabla u |^
  t \, dy \right) ^ { 2/t} )  ]. 
\end{split}
\end{equation*}
We let $ s = \max ( t, 2n/( n+2))$ and observe that our condition that
$ p > 2 ( n -1) /n$ implies that $ t < 2$. We choose $ \theta$ so that
$ 1/s = ( 1- \theta ) / 1 + \theta /2$ and use H\"older's inequality
and Young's inequality to obtain that
\begin{eqnarray*} 
\left ( \average_ { \locdom x \rho } |\nabla u |^ 2 \, dy \right ) ^ { 1/2} 
& \leq &  C (  ( 1- \rho/r) ^ {-\frac  { 2n-2}{  1- \theta} }
  \average_ { \locdom x r}  |\nabla u |\, dy +
 P _ r ( |f|^ p ) ^ { 1/p }  (x) )    \\
&& \qquad    + \frac  1 2 \left ( \average _ {
    \locdom x r } |\nabla u | ^ 2 \, dy \right ) ^ { 1/2 } 
\end{eqnarray*}
Now a standard argument as in 
 Giaquinta \cite[p.~161]{MG:1983} gives 
(\ref{RHClaim}). 
\end{proof}

\section{Atomic Estimates}
\label{Atoms}
Our next step is to estimate the gradient of solutions at the boundary and
eventually to obtain estimates for the  non-tangential maximal function $
\nontan{  (
\nabla u ) } $. It is this part of the argument that requires the
integrability conditions on $ \delta$ given in (\ref{BIntegral}) and 
(\ref{Integral}). 

Our main result is  Theorem \ref{Hardy}.  A key step in the proof,
Lemma \ref{Decay}, uses estimates for the Green function from the work
of Taylor, Kim, and Brown \cite{TKB:2012}.  We are only able to prove
these estimates in two dimensions.

Throughout this section, we assume that $\Omega$ is a Lipschitz domain
in $ \reals ^2$, that $D \subset \partial \Omega$ is an open set
satisfying (\ref{Corkscrew}), 
and that $L$ is the Lam\'e operator with coefficients satisfying
(\ref{Elliptic}).
If $ q_0$ is as in Theorem \ref{RH}, we
choose $ \epsilon$  so that $ 1< q _0 ( 1-\epsilon )/ ( 2- \epsilon )$
and require that $ \Omega$ and $D$ satisfy  (\ref{BIntegral}) and (\ref{Integral})  for this $
\epsilon$. 

Given a boundary ball $ \sball x r$, we let $ \Sigma _ 0 = \ssb x r =
\ball x r \cap \partial \Omega $ and
for $ k \geq 1$, we let $ \Sigma _k = \ssb x { 2 ^ k  r} \setminus \ssb x {
  2 ^ { k -1 } r }$. We will use this notation in the following proof
and in a similar argument in Appendix \ref{RegularityProblem}.   

\begin{theorem} 
\label{Hardy} 
If $u$ is a weak solution of the mixed problem (\ref{WMP}) with  $f=0$
and 
traction   data $f_N =a $ 
an  atom supported in $\sball x r$, then   for  $ p$ with 
$   p < q_0 ( 1-\epsilon ) / ( 2- \epsilon )$, we have 
\begin{equation}
\label{HardyEstimate}
\| 
\nontan { ( \nabla u )}
\|_{ L^ p ( \Sigma _k )} \leq  C 2 ^ { -k \gamma } (2^k r) ^ { \frac 1
  p -1 } , \qquad k \geq 0  . 
\end{equation}
Here, $ \gamma $ is the exponent of H\"older continuity for the Green
function for our mixed problem (see \cite{TKB:2012} or the estimate
(\ref{GDecay}) below). 
In addition, $\nabla u$ has non-tangential limits a.e. on $ \partial
\Omega$.

Let $ p _1 =1/( 1+ \gamma)$ then for $ p _1 < p \leq 1$, we have
\begin{equation}
\label{Hardy2}
\| \nontan{(\nabla u )} \|_ { L^ p ( \partial \Omega)} \leq C \sigma (
\sball x r ) ^ { \frac 1 p -1 } .
\end{equation}

The constants in these estimates depend on $p$, $M$ and the global
properties of $  \Omega$. 
\end{theorem}

We recall the following results of Dahlberg, Kenig, and Verchota
\cite{DKV:1988}  which are a consequence of the
existence theory for the $L^2$-traction and regularity problems. 
These results do not depend on $D$ and hold in all dimensions.
To state these estimates, we will use 
 the {\em  truncated non-tangential maximal function }defined  by 
$$
\nontan { v_ r} (x)  = \sup _ { y \in \Gamma (x) \cap { \ball x r } } |\nabla 
v ( y ) | , \qquad x \in \partial \Omega .
$$

\begin{lemma}[\cite{DKV:1988}]
\label{LocalSolve} 
Let $ x \in \partial \Omega$ and $ r  \in (0, 25 r_0)$. Suppose that  $u$ is a
solution of $ Lu = 0$ in $ \locdom x { 4r} $ and  $ u $ lies in the
Sobolev space $\sobolev   2 1 ( \locdom x { 4r } )$. 

If $ \partial u/ \partial \rho    \in L^ 2 ( \sball x { 4r} )$,
then
$$
\int _ { \sball x r }{ \nontan { ( \nabla u ) _ r }}  ^ 2 \, d\sigma 
\leq C \left ( \int  _ \sball x {  4  r } \left | \frac { \partial u } {
  \partial \rho } \right| ^ 2 \, d\sigma + \frac  1  r \int _ {
  \locdom x { 4 r} } |\nabla u |^ 2\, dy \right ) .
$$

If $ \nabla _t u $ lies  in  $L ^ 2 ( \sball x { 4r} ) $, then 
$$
\int _ { \sball x r }{ \nontan { ( \nabla u ) _ r }}  ^ 2 \, d\sigma 
\leq C \left ( \int  _ \sball x {  4  r }  |  \nabla _ t u  | ^ 2 \, d\sigma +
\frac  1  r \int _ {   \locdom x { 4 r} } |\nabla u |^ 2\, dy \right )
.
$$

In each case, $ \nabla u$ has non-tangential limits a.e. on $\sball x r$.
The constants in these estimates depend only on the dimension and
$M$. 
\end{lemma}

These estimates may be proven using arguments from \cite{DKV:1988}. See
Mayboroda and Mitrea  \cite{MR2181934} for results in two dimensions.

We give a notion of Whitney decomposition on the boundary of a
Lipschitz domain. To pass from $ \reals ^ n$ to a  more general set, 
we sacrifice the property that an open set is decomposed into disjoint
sets and 
 only require
that each point lie in at most finitely many elements of the
decomposition. This seems to be enough for our applications. 
Given  an open  set $ { \cal O } \subset \partial \Omega$, a {\em  Whitney
decomposition }of $  \cal  O$ is a 
collection of boundary balls $ 
\{ \Delta _ j = \sball  { x _j } { r_j }\}_ { j =1 } ^ \infty  $ so
that
1) ${\cal O } = \cup _ j \Delta _ j $, 2) for a constant $C_1$ that
may be chosen as large as we like,  
we have $ 
  C_1 r _ j  \leq  \dist ( \Delta _j , \partial \Omega \setminus {
  \cal O}) \leq  2C_1  r _j   $, 3) $\chi _ { \cal O}\leq  \sum _ j \chi _ { \Delta _j  } 
\leq c  _n \chi _ { \cal O } $,  
4) 
 $ \sum _j  \chi _{ \locdom { x_ j } { 4 r_j }} \leq C$. 
We leave it as an exercise to construct this decomposition. 

\note{ 
Note RB changed the constants in 2, again. 

Can we find a reference for this? A proof is in RB's notebook
  21
page 49.  Perhaps something similar is buried somewhere in the
literature on spaces of homogeneous type.  

The result can also be proved by carrying out a standard Whitney
decomposition in each coordinate cylinder. }

\begin{lemma}
\label{Whitney}
Suppose $u$ is a weak solution of the mixed problem (\ref{WMP}) with $ f_N$
in $L^ 2(N)$ and  $ f=0$.  Then for $ x\in \partial
\Omega$, $ r \in ( 0, 50r_0)$ and $ \rho \in \reals $, we have
$$
\int _ { \sball x r } {\nontan { ( \nabla u ) _{c\delta_ r}}}^ 2
\delta _ r ^ { 1- \rho}\,  d\sigma
\leq C  \left( \int  _ { \sball x { 2 r } } |f_N |^ 2 \delta _r ^ { 1- \rho}  \, d\sigma
+ \int _ { \locdom  x { 2 r}} |\nabla u |^ 2 \delta _r ^ { - \rho } \,
dy  \right ) .
$$

The constant in this estimate depends only on $M$, the dimension $n$,
and $ \rho$. 
\end{lemma}

\begin{proof} We   construct  a  Whitney decomposition  of $ \partial
\Omega \setminus \Lambda $  and let $ \{ \Delta _j = \sball { x_j }
{ r _j } \} $ denote the  boundary balls from the Whitney decomposition
which
intersect $ \sball x r $. This  gives  a family of
  boundary balls so that for each $j$, 
$ \sball x r \setminus \Lambda  
\subset \cup _ j \Delta _ j \subset \sball x { 2 r} $, 
$\sball {x_j}  { 2 r _j } \subset N$, or $ \sball {x_j} { 2 r_ j } \subset
  D$, $ r _ j \approx \delta _ { r} ( y ) $ for $y \in \locdom x {
    4r_j}$, and $ \sum \chi _ {\locdom x { 4 r_j }  } \leq C(M,n)$. 
We apply the estimate of Lemma \ref{LocalSolve}, sum on $j$, and use
the properties of the Whitney decomposition to obtain the Lemma. 
\end{proof}

The next lemma takes us from estimates in a  weighted $L^2$ space 
to  estimates in an unweighted $L^p$ space. Recall that $ \epsilon $
appears in our hypotheses (\ref{BIntegral}) and (\ref{Integral}) and
is required to satisfy $ 1< q_0 ( 1-\epsilon ) / ( 2 - \epsilon)$ with
$ q _0$ as in Theorem \ref{RH}. 

\begin{lemma} 
\label{Boundary}
Let $u$ be a weak solution of (\ref{WMP}) with $ f_N$ in $L^
\infty(N)$ and $f=0$. 

Let $p _ 0 = q _ 0 ( 1- \epsilon ) / ( 2 - \epsilon )$ and suppose
that $ 1 <  p < p_0$. For each boundary ball $ \sball x r $  with $ 0<
r < 50r_0$. 
$$
\left( \average _ { \sball x r } {\nontan  { ( \nabla  u ) _ { c \delta_r
    } } }  ^ p \, d\sigma \right) ^ { 1/p } 
\leq C (  \average _ { \locdom x { 2 r } } |\nabla u | \, dy 
+  \| f_ N \| _ { L ^ \infty ( \sball x { 2r } ) } ) .  
$$

The constant in this local estimate depends on $M$ and the constants
in (\ref{BIntegral}) and (\ref{Integral}). 
\end{lemma}

\begin{proof} 
We choose $p$ and $q$ with $ 1 < p < q ( 1- \epsilon ) /( 2- \epsilon)
< q _ 0 ( 1- \epsilon ) /( 2- \epsilon)$. We then fix $ \rho$ so that 
$( 2- \epsilon ) ( 1 - ( 2/q) )> \rho > ( 2-\epsilon ) - 2 ( 1-
\epsilon) / p$. We  apply H\"older's inequality with
exponents $ 2/ p$ and $ 2 / ( 2- p ) $,  the property
(\ref{BIntegral}), the lower bound for $\rho$ and  Lemma \ref{Whitney} to
obtain 
\begin{equation}
\label{Boundary1}
\begin{split}
\left ( \average _ {\sball x r } 
{\nontan{ ( \nabla  u ) _ { c\delta _ r } }}^ p \, d\sigma \right) ^ { 1/
  p}
 \leq &
\left ( \average _ {\sball x r } 
{\nontan{ ( \nabla u) _ { c\delta _ r } }}^ 2\delta _ r ^ { 1-\rho}  \, d\sigma \right) ^ { 1/2} 
\left ( \average _ { \sball x r } \delta _r ^ {( \rho -1 ) p / ( 2-p)
} 
\,d\sigma \right ) ^ { 1/ p - 1/ 2 } \\
\leq&
C r ^ { (\rho -1)/2 }  \left ( \average _ { \sball x r } 
{\nontan{ ( \nabla u ) _ { c\delta _ r } }}^ 2 \delta _r ^ { 1- \rho }
\,  \, d\sigma \right) ^ { 1/2} 
\\
\leq  &
C[ r ^ { \rho /2 } \left ( \average _ { \locdom x { 2 r} } |\nabla u
|^ 2 \delta _ r ^ { - \rho } \, dy  \right)^ { 1/2}  
 \\
& \qquad  
+ r ^ { ( \rho -1 ) / 2} 
\left ( \average _ { \sball x { 2 r}} |f_N |^ 2 \delta _r ^ { 1 - \rho } \, d\sigma \right) ^ { 1/2 }  ] .
\end{split}
\end{equation}

To estimate the integral over $ \Omega_ { 2r } ( x) $, we use
H\"older's inequality with exponents $ 2/q$ and $ q / ( q-2) $, the
property (\ref{Integral}), the upper bound  for $\rho$,  and
 Theorem \ref{RH}  to
obtain 
\begin{equation}
\label{Boundary2}
\begin{split} 
r ^ { \rho /2 } \left ( \average _ { \locdom x { 2r} } |\nabla u | ^ 2
\delta _ r ^ { -\rho } \, dy \right ) ^ { 1/2}  
& \leq C \left ( \average _ { \locdom x { 2r} } |\nabla u |^ q \, dy
\right ) ^ { 1/ q} \\
& \leq C [\average _ { \locdom x { 4r} } |\nabla u | \, dy + \left (
\average _ { \sball x { 4r} } |f_N | ^ { q ( n-1 ) /n  }\, d\sigma 
\right) ^ { \frac  n { q(n-1 ) } }  ]. 
\end{split}
\end{equation}
The Lemma follows from (\ref{Boundary1}),     (\ref{Boundary2}), and
a simple covering argument to replace $  4r$ by $ 2r$ on the right hand
side of the estimate.
\end{proof}

Before we proceed to the proof of Theorem \ref{Hardy}, we recall a
version of the energy estimate for solutions of the weak mixed
problem, (\ref{WMP}). 

\begin{lemma}
\label{Energy}
Let $u$ be a solution of the weak mixed problem (\ref{WMP}) with data
$f_N =a$, an atom for $N$ and $ f=0$. 
The solution $ u$ satisfies 
$$
\int _ { \Omega} |\nabla u |^ 2 \, dy \leq C.
$$

The constant $C$ depends on the global character of $\Omega$. 
\end{lemma}

\begin{proof} According to (\ref{BoPo}), we have that $ \sobolev 2 {
    1/2} _D ( \partial \Omega) \subset BMO( \partial \Omega)$, the
  space of functions of bounded mean oscillation on $ \partial \Omega$. Thus if
  $a$ is an atom for $N$ the map $ u \rightarrow \int_ { \partial
    \Omega}  a^ \beta u^
  \beta \, d\sigma $ 
  lies in $ \sobolev 2 { -1} _ D ( \Omega)$. The estimate of the Lemma
  follows. 
\end{proof} 

We are only able to prove the following Lemma in two dimensions.
Finding an appropriate substitute for this Lemma is the main obstacle
to studying the mixed problem for the Lam\'e system in higher
dimensions.

\begin{lemma} 
\label{Decay} 
Let $a$ be an atom for $N$ that is supported in a  boundary ball $ \sball x r
$. We may find an exponent $ \gamma > 0$ which depends only on $M$ so
that if  $u$ is a solution of the weak mixed
problem (\ref{WMP}) with $f_N =a$ and $f=0$, then with $p$ as in 
Theorem \ref{Hardy},  we have  a constant $C$  so
that 
\begin{equation}
\label{DecayEstimate}
\left (\int _ {   \Sigma _k  } |\nabla u |^ p \, d\sigma \right ) ^
 { 1/p}  
\leq  C 2 ^ { - k \gamma } (2 ^ k r )^{ \frac 1 p -1  } , 
\qquad k \geq 0. 
\end{equation}

The constant in the estimate depends on the global character of
$\Omega$. 
\end{lemma}

This Lemma gives the estimate of Theorem \ref{Hardy} for  $\nabla
u$. Additional work is needed to obtain the estimate for the
non-tangential maximal function.  

\begin{proof} We fix a boundary ball $ \sball x r $ and let $a$ be an
atom for $N$ that is supported in $\sball x r$. We let $u$ be a
solution of the weak mixed problem (\ref{WMP}) with $ f_N =a$ and $
f =0$.  We apply Lemma \ref{Boundary} on a boundary ball  $ \sball y
R$  and use the  normalization of the atom  to obtain that 
$$
\left ( \average _{\sball y R } |\nabla u |^ p \, d\sigma \right) ^ { 1/ p}
\leq C ( \average _ { \locdom y {2R}} |\nabla u | \, dz + \sigma (
\sball x r ) ^ {
  -1} ) . 
$$
When $ n =2$, the Cauchy-Schwarz  inequality and  the energy estimate
of Lemma \ref{Energy} gives that for any local domain $\locdom y R $ 
$$
\average _ { \locdom y R } |\nabla u | \, dz \leq \left ( 
\average _ { \locdom y R } |\nabla u | ^ 2 \, dz \right ) ^ { 1/2} 
\leq C/R.
$$
Together, these observations imply that  
$$
( \average _ { \sball y R } |\nabla u | ^ p \,d\sigma ) ^ { 1/ p }
\leq C ( 1 / R + 1/ r) 
$$
and if we set $ y =x$, we obtain  the estimate (\ref{DecayEstimate})
but with a constant that depends on $k$.  

We fix $k _ 0$ so that we have $ \dist (\sball x r , \Sigma _ {k_0} )
\approx r$ and establish (\ref{DecayEstimate}) for $k \geq k _0$ with
constant that does not depend on $k$.
Towards this end, we consider a boundary ball $\sball y R$ with $
\sball y { 8R} \cap \sball x r = \emptyset$. Since $ a = 0$ on $
\sball y { 2R}$, the estimate of Lemma \ref{Boundary} implies that
$$
( \average _ { \sball y R } |\nabla u | ^ p \, d\sigma ) ^ { 1/ p } 
\leq C \average _ { \locdom y { 2R } } |\nabla u | \, dz .
$$
The local ellipticity condition (\ref{LocalElliptic}) implies a
Caccioppoli inequality 

\begin{equation}
\label{Cacc}
\int _ { \locdom y {2R} } |\nabla u | ^ 2 \, dz \leq \frac C { R ^ 2}
\int _ { \locdom y { 4 R} } |u|^ 2 \, dz 
\end{equation}
provided that $u$ is a solution of (\ref{WMP}) with $ f_N =0 $  on $
\sball y {4 R}$ and $ f=0$ on $\locdom y {4R}$. Next we note that the Green function estimates of
Taylor, Kim, and Brown \cite{TKB:2012} and the arguments in
\cite{OB:2009} imply that we have a constant $C$ and exponent $ \gamma
> 0$ so that 
\begin{equation}
\label{GDecay}
|u(y) | \leq C\left ( \frac r { |x- y| } \right ) ^ \gamma  , \qquad y
\in \Omega \setminus \locdom x { 2 r} . 
\end{equation}
Combining the estimate of Lemma \ref{Boundary}, the H\"older inequality,
(\ref{Cacc}), and (\ref{GDecay}) gives that 
$$
\left ( \average_ { \sball y R } |\nabla u | ^ p \, d\sigma \right) ^ {
  1/ p } \leq \frac C R \left ( \frac r R \right ) ^ \gamma . 
$$
We may cover $ \Sigma _k $ by boundary balls with radius comparable to
$ 2 ^ k r$ and use the above estimate to obtain the estimate of the
Lemma for $ k \geq k _0 $. 
\end{proof}


Now we are ready to give the proof of Theorem \ref{Hardy}. 

\begin{proof}[Proof of Theorem \ref{Hardy}]
We let $ \phi\in C^ \infty ( \bar \Omega)$ be a smooth function  which
is zero outside $ \locdom x {4r}$ and is zero in a neighborhood of $
\Lambda$. We let $ \Gamma ^ { \alpha \beta}$ be the matrix fundamental
solution of the Lam\'e 
operator in $ \reals ^2$, thus $ (L\Gamma ^ { \cdot \gamma }) ^ \alpha
= \delta _ {\alpha \gamma} \delta _0$ where $ \delta _ { \alpha \gamma
} $ is the Kronecker delta 
and $ \delta _0$ is the Dirac delta measure at 0. 
We claim the following representation formula for derivatives of $u$,
with $ \alpha , k =1,2$, 
\begin{equation}
\label{GradRep}
\begin{split}
\phi(x) \frac { \partial u ^ \alpha }{\partial x _k }( x)
= &\int _ { \Omega } \frac \partial { \partial y _ i } a^ {ij}_ { \alpha \beta}
\frac { \partial \Gamma ^ { \beta \gamma } }  { \partial y _ j}  ( x- \cdot) 
\phi \frac { \partial u ^ \alpha } { \partial y _ k } \, dy  \\
= & \int _ { \partial \Omega }  \nu _ i  a ^ { ij } _ { \alpha \beta} 
\frac { \partial  \Gamma ^ {\beta\gamma} } { \partial y _ j }  
(x-\cdot) \phi \frac { \partial u ^ \alpha } { \partial y _k }
- \nu _ k  a ^ { ij } _ { \alpha \beta} 
\frac { \partial  \Gamma ^ {\beta\gamma} } { \partial y _ j }  
(x-\cdot) \phi \frac { \partial u ^ \alpha } { \partial y _i }  \\
& +\nu _ j  a ^ { ij } _ { \alpha \beta} 
\frac { \partial  \Gamma ^ {\beta\gamma} } { \partial y _ k }  
(x-\cdot) \phi \frac { \partial u ^ \alpha } { \partial y _i }
\,d\sigma 
- \int _ { \Omega } a ^ {ij}_{\alpha \beta} \frac {\partial  \Gamma ^ { \beta \gamma} } {\partial y_j} (x-\cdot)
\frac { \partial \phi }{ \partial y _i } \frac { \partial u ^ \alpha }{\partial y _k}  \\
&-a ^ {ij}_{\alpha \beta} \frac {\partial  \Gamma ^ { \beta \gamma} }   {\partial  y _j } (x-\cdot)
\frac { \partial \phi }{ \partial y _k } \frac { \partial u ^ \alpha }{\partial y _i}  
 +a ^ {ij}_{\alpha \beta} \frac {\partial  \Gamma ^ { \beta \gamma} }   {\partial  y _k } (x-\cdot)
\frac { \partial \phi }{ \partial y _j } \frac { \partial u ^ \alpha }{\partial y _i} 
\, d\sigma. 
\end{split} 
\end{equation}
Of course, the integral to the right of the first equals sign  should be
interpreted as the Dirac delta acting on a smooth function. 
If $u$ is smooth, the proof of  (\ref{GradRep}) is accomplished by integration
by parts. We begin by moving the $ \partial / \partial y _i$ derivative,
then the $\partial/\partial y _k$ derivative, and finally the 
$\partial/\partial y _j$ derivative. Since Lemma \ref{LocalSolve} gives non-tangential
maximal function estimates for $u$ and  $ \nabla u $ away from $ \Lambda$ and we
assume that $ \phi$ is supported in a coordinate cylinder and is zero in a 
neighborhood of $ \Lambda$,  standard limiting arguments allow us to 
obtain the representation formula  (\ref{GradRep})  without the
assumption that $u$ is smooth in $ \bar \Omega$. 

Next, we remove the restriction that $ \phi$ vanish  near $ \Lambda$. 
Let 
$ \psi _ \tau $ be a function which is one when $ \delta (x) > 2 \tau$,  zero
if $ \delta (x) < \tau$, and satisfies $| \partial ^ \alpha    \psi _ \tau
/\partial x^ \alpha | \leq C_ \alpha / \tau ^ {|\alpha |}$. 
Replace $ \phi$ in (\ref{GradRep}) by $ \psi _ \tau \phi$ where $
\phi$ is a smooth function which is zero outside $ \locdom x { 4r}$,
but is not required to vanish on $ \Lambda$.  Using H\"older's
inequality, the terms involving derivatives of $ \psi_\tau$  
in (\ref{GradRep}) have the upper bound 
\begin{equation} 
\label{Collar}
\int _ { \Lambda _ { \tau, r} } |\nabla \Gamma ( x- \cdot)| |\nabla u
| |\phi \nabla \psi _\tau | \, dy
\leq \frac {C_x} \tau \left( \int _ {
\Lambda _ {\tau, r} } |\nabla u | ^ q \, dy \right ) ^ { 1/q} 
|\Lambda _ {\tau ,r } |^ { 1/q'} 
\end{equation}
where $ \Lambda _ { \tau, r} = \locdom x {4r} \cap \{ y : \tau <  \delta ( y ) < 2 \tau \}$
and $q$ is chosen to satisfy  $ ( 2- \epsilon ) / ( 1- \epsilon ) < q < q_ 0$. From the
estimate (\ref{Integral}), we obtain the upper bound
$$
|\Lambda _ {\tau, r} |\leq C \tau ^ a \int _ {\locdom x { 4r}} \delta ^ {-a} \, dy  \leq C _r \tau^ a
$$
provided $ 0<  a < 2 -\epsilon$. Our choice of $q$ implies $ q ' < 2 -\epsilon$ 
so we set $a = q'$. Then the estimate (\ref{RHI}) for $ \nabla u$ implies that 
the right-hand side of (\ref{Collar}) tends to zero with $\tau$. From
the estimate of 
Lemma \ref{Decay}, we have $\nabla u$ is in $L^ p ( \partial \Omega)$ for some $p\geq
1$, thus we may use the dominated convergence theorem  and let $ \tau
$ tend to zero in the  
boundary terms of (\ref{GradRep}).  This removes the restriction  that
$\phi$ vanish in a neighborhood of $ \Lambda$.

To establish (\ref{HardyEstimate}), we begin with the representation formula
(\ref{GradRep}).  We let $ \phi$ run over a
partition of unity and sum to obtain the representation 
$$
\frac { \partial u ^ \alpha } { \partial x _ k } ( x) =
\int _ { \partial \Omega } a ^ { ij } _ {\alpha \beta  } \frac {
  \partial \Gamma ^ { \beta \gamma } } { \partial y _j }  (x - \cdot ) 
( \nu _ i \frac { \partial  u ^ \alpha } { \partial y _ k } - \nu _ k
  \frac { \partial u ^ \alpha } { \partial y _ i } ) 
+ \frac { \partial \Gamma  ^ { \beta \gamma } } { \partial y _ k } (
x- \cdot ) \nu _ j a ^ { ij }_ {\alpha  \beta  } \frac { 
\partial u ^
  \alpha } { \partial y _ i } \, d\sigma , \alpha,k = 1, 2.  
$$
The divergence theorem implies that 
\begin{gather}
\label{MZ1} 
\int _ { \partial \Omega } \nu _ i a ^ { ij } _ { \alpha \beta } \frac
     { \partial u ^ \beta } {\partial y _ j } \, d\sigma = 0, \qquad \alpha
     =1,2 \\
\label{MZ2} 
\int _ { \partial \Omega } \nu_ i \frac { \partial u ^ \beta } {
  \partial y _k } -  \nu _ k \frac { \partial u ^ \beta } { \partial y _
  i } \,d \sigma = 0 , \qquad i, k, \beta = 1,2
\end{gather}
where the techniques used to prove  (\ref{GradRep})
help to justify the application of the
divergence theorem. We may use the  theorem of Coifman, McIntosh,
and Meyer \cite{CMM:1982}, the observations  (\ref{MZ1}-\ref{MZ2}),  the
size estimates of Lemma \ref{Decay}, and the notion of molecule (see
\cite{CW:1976}) to 
establish the estimates (\ref{HardyEstimate}). 
For $ p > 1/ ( 1+ \gamma)$, the estimate (\ref{Hardy2}) follows easily from
(\ref{HardyEstimate}). 
\end{proof}

\section{Uniqueness}
\label{JustOne}

We take advantage of the restriction to $n=2$ to give a different
proof of uniqueness than found in \cite{OB:2009,TOB:2011}. The argument uses a
version of the Hardy-Littlewood-Sobolev  theorem on fractional
integration.  The
proof is adapted from one found in \cite{RB:1995a} and \cite[Lemma
  11.9]{MW:2011} where a result similar to Lemma \ref{FI} is
established  in
dimensions $n \geq 3$.  The result for $n=2$ is  more technically
involved.

\begin{lemma}
\label{FI} 
Let $L$ be the Lam\'e operator with coefficients satisfying
(\ref{Elliptic}). 
Let $ \Omega\subset \reals ^2$ be a Lipschitz domain and suppose that
$u$ is a solution 
of $L u =0$ with $  {\nontan  { ( \nabla u )} }  \in L^ p ( \partial
\Omega)$. If $p< 1$, $q$ is defined by $1/q=1/p-1$, and $K$ is a
compact subset of $ \Omega$, then we have 
$$
\| \nontan u \| _ { L^ q ( \partial \Omega)}
\leq C ( \| \nontan  {{ ( \nabla u )}} \|_  { L ^ p ( \partial \Omega)}  
+ r _ 0 ^ { 1/p -1} \sup _ { K } |u | ) . 
$$
The constant in this estimate depends on $p$, $K$ and the global
character  of $\Omega$. 
 \end{lemma}

 For this argument it will be convenient to use 
$ \ssb x r = \ball x r  \cap \partial \Omega$ for $ x\in \partial  \Omega$ 
 in place of $ \sball x  r$ and to define the 
Hardy-Littlewood maximal operator on the boundary by 
$$
\hlmax f(x) = \sup _ { r > 0 } \average _ { \ssb x r } |f| \, d\sigma .
$$

\begin{proof}
We let $ \{ \Omega _k \}$ be a sequence of approximating domains with
$ \bar \Omega _k \subset \Omega$  as
in Verchota \cite{GV:1982}. We will show that we have 
\begin{equation}
\label{FI0} 
\| \nontan u \| _ { L^ q ( \partial \Omega_k)} 
\leq C ( \| \nontan { ( \nabla u ) } \| _ { L^ p ( \partial \Omega_k)}
+ r _ 0 ^ { 1/ q } \sup _ { K } |u | )
\end{equation}
where $p$ and $q$ are as in the  statement of the Lemma  and the constant $C$ is
independent of $k$. Letting $k $ tend to $\infty$ gives the result for
$ \Omega$. The advantage of working on the subdomain $ \Omega_k$ is
that we may assume that $ \| \nontan u \| _ { L^ q ( \partial \Omega_k
  )}$ is finite. In what follows, we drop the subscript $k$ and assume
that 
$\|\nontan u \| _ { L ^ q ( \partial \Omega)} $ is finite. 

To begin we observe that if $ y \in \Gamma (x) $,
for some $ x\in \partial  \Omega$, then for $ \epsilon >0$, and $ \beta $ a
multi-index, we have 
\begin{equation}
\label{FI1}
| \frac { \partial ^ \beta u } {\partial y ^ \beta } (y)  |
\leq \frac C { d(y) ^ { |\beta |} } 
\left ( \average _ {S(x,y) } { \nontan{ u }}  ^ \epsilon  \, d\sigma
  \right ) ^ { 1/ \epsilon }.
\end{equation}
Here, $ S(x,y) = \{ z : |z-x| < ( 2 + 2 \alpha ) d(y) \} \cap \partial
\Omega$ with $ \alpha $ as in the definition of the approach regions $
\Gamma (x)$ and recall that $ d(y) = \dist(y, \partial \Omega)$. To
establish (\ref{FI1}), choose $ \hat x \in \partial 
\Omega$ so that $|\hat x - y | = d(y )$. We claim that if $ |w -y |
< \min \{ \alpha / 4 , \alpha / ( 2 + 2\alpha )\} d(y) $ and $|z- \hat x| <
\alpha d(y ) / 4$, then $ w\in \Gamma (z) $. To establish this claim,
first note that $ d( y) \leq d(w) + | w-y |  \leq  d(w) + \alpha d( y
) /( 2 + 2 \alpha )$ and hence $ d ( y ) \leq  d(w)  ( 2 + 2 \alpha) /
( 2 + \alpha )$. Using the triangle inequality and our choices for $
z$ , $ \hat x$, and $w$, we have $|z-w | \leq |w-y | + |y - \hat x | +
|\hat x -z | < ( 1 + \alpha /2 ) d( y) \leq ( 1+ \alpha ) d(w) $. 
Also, if $z \in \ssb { \hat x}  { \alpha d( y )
  /4}  $, then by the triangle inequality $ |z-x | \leq |z- \hat x | + |\hat x -y |
+ |y - x | < ( 2 + 2 \alpha ) d( y)$ and we  conclude that 
  $ \ssb { \hat x } { \alpha d( y ) / 4 } \subset \ssb x  { ( 2 + 2
  \alpha ) d ( y ) } $. 
 Thus, with $ r = \min \{ \alpha
/4, \alpha / ( 2 + 2 \alpha )  \} d( y)$, we may use interior estimates
for derivatives of $u$ to obtain
$$
 |  \frac { \partial ^ \beta  u } { \partial y ^ \beta } (y)
  | 
\leq \frac  C { d(y ) ^ { |\beta  | } } \sup _{ w \in \ball y r } |u
(w) |
\leq \frac C { d( y ) ^ { |\beta |} } \left ( \average _ { S(x,y ) } 
{   \nontan  u  } ^ \epsilon  \, d\sigma \right ) ^ { 1/ \epsilon } . 
$$
The second inequality follows since $ \ball y r \subset \Gamma (z) $
for all
$ z $ in $\ssb  x { \alpha d(y) / 4 }  $  and hence $ \ssb x {\alpha
  d(y) } \subset S(x,y)$.  We have established
(\ref{FI1}).

To begin the proof of (\ref{FI0}), we let 
 $\{ Z _ i = \cyl { x _ i } { r _ 0 }\} _ { i = 1 } ^ N$ be the
covering 
of $ \partial \Omega$ by coordinate cylinders  guaranteed by the
definition of a Lipschitz domain and recall that we are in $ \reals ^
2$. 
We let $ T _ i = \{ ( y _1, y _2 ) : |y _ 1 - x _ {i,1} |  \leq r _
0  ,  y _ 2 = x _ {i,2}   + ( 4M +2 ) r _ 0 \}$ denote the ``top'' of
each cylinder $Z_i$ and then let $ K ' $ denote the compact set $
\cup_i T _i $. Applying the 
estimate (\ref{FI1}) to $ \nabla u$ with $ \beta =0$ and $ \epsilon =
p$ gives that for any pair $K$ and $J$  of  compact subsets of $ \Omega$,  we have
\begin{equation}
\label{FI11}
\sup _ { J } |u|\leq  \sup _K |u | + C r _ 0 ^ { 1- 1/p } \|
\nontan {(\nabla u)} \|_{ L^p( \partial \Omega)}.
\end{equation}

The heart of the proof is to establish for $ \epsilon > 0$, $ \theta
\in ( 0, 1)$, and  $x \in \partial \Omega$,  that
\begin{equation}
\label{FI2}
|\nontan u (x) | \leq C ( \hlmax ( { \nontan u }^ \epsilon) ( x) ^ {
  \frac { 
  1- \theta }  \epsilon }\,  I
_ \theta ( \nontan { (\nabla u ) }) ( x)^\theta  +  \sup _{K'} |u | +
 r_ 0 ^ {1 - \frac 1 p } \| \nontan { ( \nabla u ) }\|_{L^ p ( \partial \Omega)} )
\end{equation}
where $ I_ \theta$ is the fractional integral defined by 
$$
I _ \theta f(x) = \int _ { \partial \Omega } f(y ) | x- y | ^ { \theta
  -1 }\, d\sigma.
$$

We accept the claim (\ref{FI2}) for the moment and complete the proof
of the Theorem. We raise both sides of (\ref{FI2}) to the power $q$,
integrate over $ \partial \Omega$, and apply the H\"older inequality 
 with exponents $
1/ \theta$ and $ 1/(1-\theta)$ to obtain 
$$
\| u ^ * \|_ { L^ q( \partial \Omega)} \leq C ( \| \hlmax ( { u^ {*
    \epsilon} } ) \| _ { L^ {q/\epsilon}( \partial \Omega)} ^
   {(1-\theta)/\epsilon}
 \| I _ \theta ( \nontan{ (\nabla u )})\| _ { L^ q ( \partial \Omega )
 } ^ \theta   + r _ 0 ^ { 1/q} \sup _ { K'} |u| + \| \nontan { (
   \nabla u )} \| _ { L^p ( \partial \Omega )} )
$$
We choose $ \epsilon $ so that  $ q/ \epsilon > 1$  and hence  the maximal
operator $\hlmax$ is bounded on $ L^ { q/\epsilon } ( \partial
\Omega)$, choose $ \theta $ with $ 0 < \theta < p$   so  that $ I _
\theta : L ^ { p / \theta } ( \partial \Omega) \rightarrow  L ^ { q /
  \theta } ( \partial \Omega ) $ with $ \theta/ q = \theta (1 / p -
1) $.  With these choices, we may use Young's inequality, our {\em a
  priori }assumption that $ \nontan u \in L^ q ( \partial \Omega)$ and
(\ref{FI11}) to obtain the Theorem.

Finally, we establish the key estimate (\ref{FI2}). We fix $ x\in
\partial \Omega$ and consider $ y \in \Gamma (x) $ and suppose that
$y$ lies in a coordinate cylinder $ Z_i$. If $ \alpha $ is large
enough, we have that the line segment $\{ y + te_2 : y _ 2 < t < r _0
\}$ lies in $\Gamma(x)$. (As usual, we are using the coordinate
system  for the  coordinate cylinder $Z_i$ to state this condition.) 
Thus, we may use the fundamental theorem of calculus to write 
\begin{equation}
\label{FI5}
|u (y) | \leq \int _ { y _ 2 } ^ { r _ 0 } |\frac { \partial u }{
  \partial y _ 2 } ( y_ 1, t ) |  \, dt 
+ \sup _ { K ' } |u | .
\end{equation}
We apply the estimate (\ref{FI1}) to obtain the bound 
$$
| \frac { \partial u }{ \partial y _ 2} ( y _ 1, t ) | ^ { 1- \theta } 
\leq \frac C { d ( y_1, t ) ^ { 1- \theta } } \hlmax ( { \nontan u } ^
\epsilon) (x) ^ { ( 1-\theta )/ \epsilon } .
$$
Also, if we apply estimate (\ref{FI1}) to $ \nabla u $, 
 then we obtain 
$$
|\nabla u ( y_1, t ) | ^ \theta \leq C  \average _ { S ( x, ( y
  _ 1, t ))} { \nontan { ( \nabla u) } }^ \theta \, d\sigma   . 
$$
We write $ |\partial u / \partial y _2|^ { 1-\theta +\theta} $ in
(\ref{FI5}) and use the two previous estimates to obtain 
\begin{equation}
\label{LDE}
\begin{split}
\int _ { y _2 } ^ { r_ 0 } | \frac { \partial u }{ \partial y _ 2 } (
y _ 1, t ) | \, dt 
& \leq  C \hlmax ( { \nontan u } ^ \epsilon ) ( x ) ^ { ( 1- \theta ) /
  \epsilon } \\
&\qquad \times  \int _ { y _ 2} ^ { r _ 0 } \frac 1 { d ((y _ 1 , t ))^ {
    2-\theta } } \int _ { S ( x, ( y_ 1, t )) } { \nontan { ( \nabla u )
    } }^ \theta \, d\sigma \, dt  \\
& \leq C \hlmax ( {\nontan  u } ^ \epsilon ) ( x) ^ { ( 1- \theta ) /
  \epsilon }I _ \theta ( {\nontan { (\nabla u )} }  ^ \theta ) ( x) . 
\end{split}
\end{equation}
The last step uses Tonelli's  theorem to change the order
of integration. The estimates  (\ref{FI5}) and  (\ref{LDE})  give a bound
for the supremum of $u$ in  $\Gamma (x) \cap (\cup Z_i)$. Estimate
(\ref{FI11}) allows us to estimate the supremum of $u$ in $ \Omega
\setminus \cup _ i Z _i$. Together these two observations give
(\ref{FI2}). 
This completes the proof of the estimate
(\ref{FI2}) and hence the Lemma. 
\end{proof}

\note{Many changes to the proof of the previous theorem. Read
  carefully.}

We are ready to give uniqueness for  the $L ^ p 
$-mixed problem. 
\begin{theorem} 
\label{Uniq}
Suppose $ \Omega \subset \reals ^2$ is  Lipschitz domain,  $ \Omega$ and $D$ satisfy 
(\ref{Corkscrew}), (\ref{BIntegral}), and (\ref{Integral}) with $
q_0(1-\epsilon)/ ( 2- \epsilon ) > 1$, and  $ L$
is the Lam\'e system with coefficients  satisfying (\ref{Elliptic}). 
There exist $ p_1 < 1 $ so that if $ p > p_1$  and  $u$ is a solution of
(\ref{MP}) with  $f_N =0$ and $ f_D =0$, then $u =0$. We assume that
the traction operator 
 exists in the sense of non-tangential limits  if  $p \geq 1$
and in the sense of  
(\ref{HPBV}) if $ p < 1$. 
\end{theorem}

\begin{proof}  It suffices to prove the Theorem for $p < 1$. We let
  $u$ be a solution of the mixed problem (\ref{MP}) with $ f_N=0$ and $f_D=0$. 
Our argument will proceed by duality and requires the existence of a
solution  when the traction  data is an atom for  $N$ as  given in
Theorem \ref{Hardy}. 
We let $a$ be an atom for $N$ and claim that 
\begin{equation}
\label{Uniq1}
\int _ {N } a^\beta u ^ \beta  \, d\sigma = 0 .
\end{equation}
We accept the claim and complete the proof. We first  observe that $ \nontan {( \nabla u ) } \in
L ^ p ( \partial \Omega )$ and thus  Lemma \ref{FI} implies  that $\nontan u
\in L^ q ( \partial  \Omega )$ with $ 1 / q = 1/ p -1 $. Also, it is
easy to see that if  $\nontan { ( \nabla  u ) } \in L ^ p ( \partial \Omega
 )$  then $u$ has non-tangential limits a.e. The claim
implies that these limits are zero.   From the work of
\cite{DKV:1988},  we have
uniqueness for  the  $L^q$-Dirichlet problem for $q \geq 2$ and any Lipschitz
domain and thus we have $ u=0$ if $ q\geq 2$.  
Since we are in two dimensions, we have  $q \geq 2 $ if $ p \geq 2/3$, thus we will require that $ p_ 1
\geq 2/3$.

To establish the claim (\ref{Uniq1}), we let $v$ be the solution of
(\ref{MP}) with data $ f_N =a$ and $ f_D =0$. From Theorem \ref{Hardy}
 we
have $ \nontan {( \nabla v )} \in L^ t ( \partial \Omega )$ for $ t <
p_0$.   Lemma \ref{FI} implies that $\nontan v \in L^
q ( \partial \Omega )$ for $ 1/q = 1/ t -1$. Thus if $ 1 / p _0 + 1 /
p _1 \leq 2$, then we may find $t$ so that $ 1/t + 1/p \leq 1$  and
use the dominated convergence theorem to obtain
$$
\lim _ { k \rightarrow \infty } \int _ { \partial \Omega _ k } u ^
\beta ( \frac
     { \partial v } { \partial \rho } ) ^ \beta \, d\sigma  = \int _ { N } u ^
     \beta a ^ \beta \, d\sigma .
$$
The estimates for the Green function in section 4 of \cite{TKB:2012}
imply  that there is an exponent $ \gamma$ so that the
solution $v$ lies in $ C^ \alpha _D( \partial 
\Omega)$ for $ \alpha \leq \gamma$. If $ \alpha \geq  1/ p _ 1 -1 $,
then our assumption that  
$ \partial u / \partial \rho = 0 $ in $H^ p ( N)$ implies that 
$$
\lim _ { k \rightarrow \infty } \int _ { \partial \Omega _ k } v ^
\beta ( \frac
     { \partial u } { \partial \rho } ) ^ \beta \, d\sigma = 0. 
$$
While the Green identity gives  for every $k$ that 
$$ 
  \int _ { \partial \Omega _ k } v ^ \beta ( \frac { \partial u } {
  \partial \rho } ) ^ \beta - u ^ \beta ( \frac { \partial v } {
  \partial \rho } ) ^ \beta \, d\sigma =0.
$$
The last three displayed equations imply the claim (\ref{Uniq1}). 

To summarize the conditions on $p_1$, we need $ p _ 1\geq 2/3$, $p _ 1
\geq p_ 0 / ( 2p_ 0-1)$ and $ p_ 1 \geq  1 / ( 1+ \gamma )$ to establish uniqueness. 
\end{proof}

\section{Existence of solutions} 
\label{Exist}

In this section, we give the details needed to establish the existence
of solutions for the $ L^ p$-mixed problem for $p$ in
an interval containing 1. When $ p \leq 1$, we take our data from
Hardy spaces and for $p > 1$, the data is taken from $L^p$ spaces. 
Given the estimates of Theorem \ref{Hardy}, the argument is not so
different from results in previous work of the authors \cite{RB:1995a}
and \cite{OB:2009}.

\begin{theorem} 
\label{pEx}
Let $ \Omega \subset \reals ^ 2 $  be a Lipschitz domain, suppose that
$D$ satisfies (\ref{Corkscrew}), (\ref{BIntegral}), and
(\ref{Integral}), and  that $L$ is the  Lam\'e operator 
with coefficients 
satisfying the ellipticity condition (\ref{Elliptic}). 

We may find an exponent $ p_1$ so that if $ 1 \geq p > p_1$, the
$L^p$-mixed problem has  a solution. 
This means that if $ f _N \in H^ p (N)$ and $ f_D\in H^ { 1, p } (
\partial \Omega)$, then we have a
solution to the $L^ p  $-mixed problem (\ref{MP}) which
satisfies 
$$
\| u \| _ { H ^ { 1,p}( \partial \Omega) } +  \| \frac { \partial u }
   { \partial \rho } \| _ { H^ p ( \partial 
   \Omega) }   + 
\| \nontan{( \nabla u )} \| _ { L^ p ( \partial \Omega )} 
\leq  C ( \| f _ N \| _ { H ^ p (N) } + \| f_D \| _ { H ^ { 1,p }( \partial \Omega)} ).
$$
The normal derivative of $u$ exists in the sense of  (\ref{HPBV}), 
$u$ has non-tangential limits a.e. on the boundary,  and these limits
vanish on 
$D$. 
\end{theorem}

\begin{proof}
By the results of Appendix \ref{RegularityProblem}, we may assume that $f_D=0$. 
Suppose that $f_N$ lies in $H ^ p (N)$ and choose a representation
of $ f_N = \sum _ j \lambda _ j a _j$ with the atom  $ a _ j$ supported in
$\sball  { x _ j } { r _ j } $ and $ \sum _ j \lambda _ j ^ p \sigma ( \sball {
  x_ j } {r_j } ) ^ { 1- p }  \leq 2 \| f _ N \| _ { H ^p (N)}$. We
let $ v_ j $ be the solution of (\ref{WMP}) with the traction data
$f_N$ an atom  $a_j$ 
and $f=0$ and set $ u _M = \sum _ { j = 1} ^ M\lambda _j v _j$. From
Theorem \ref{Hardy},  we have 
$$
\| \nontan { \nabla ( u_M- u _ { M' } )} \| ^p_ { L^ p ( \partial
    \Omega)} \leq C \sum _ { j = M +1 } ^ { M ' } \lambda _ j^p \sigma (
  \sball { x_ j } { r_j } ) ^ { 1- p  } , \qquad M'  >  M. 
$$
From Lemma \ref{FI}, we have that $ \nontan { u_M } \in L^ q (
\partial \Omega)$ for $ 1/q = 1/ p -1$.  Thus, the estimate of
(\ref{FI1}) implies that $u _M$ converges uniformly to a  function $u$
on compact subsets of $ \Omega$. Furthermore, we have that $u$ is a
solution of $L u = 0$ in $ \Omega$, $ \| \nontan{(\nabla u ) } \| _ {
  L^ p ( \partial \Omega)} \leq C \| f _N \| _ { H ^ p (N)}$, and $
\nabla u $ has non-tangential limits a.e. on $\partial \Omega$. 
To find $ \partial u /\partial \rho$ on $ \partial \Omega$, 
we first note that as in (\ref{MZ1}), we have  that for any $j$ and
$k$, 
$$
\int _ { \partial \Omega } \frac { \partial v_j }{ \partial \rho } \,
d\sigma  
= 0 \qquad \mbox{and} \qquad\int _ { \partial \Omega _k } \frac { \partial v _j } {\partial \rho }
\,d \sigma = 0 .
$$
The estimates  (\ref{HardyEstimate}) of Theorem  \ref{Hardy} imply
that 
\begin{equation}
\label{pEx1}
\|\frac  { \partial v _ j } {\partial \rho } \| _ { H ^ p ( \partial
  \Omega  )} + \sup _ k  \| \frac { \partial v _ j } { \partial \rho } \| _ {
  H^ p ( \partial \Omega  _k ) }  \leq C.
\end{equation}
To establish (\ref{pEx1}), we observe that 
the size estimates (\ref{HardyEstimate}) and (\ref{MZ1}) imply that
$ \partial v_j /\partial
\rho$ 
  is a molecule as in  \cite{CW:1976}  and thus can be decomposed
into  atoms. Here, we use that the atomic Hardy spaces may be  defined
with atoms from any  $L^t$ space, $t >1$, and the space we obtain is
independent of $t$. See Coifman and Weiss \cite{CW:1976} or
Mitrea and Wright \cite[sections 2.2-2.3]{MW:2011} for more information. Thus 
$
\sum _ { j =1 } ^ \infty \lambda _j  \partial v_j / \partial
  \rho
$
defines an element of the atomic Hardy space $H^ p ( \partial \Omega)$
as long  as  $ 1/ ( 1 + \gamma ) < p \leq 1$ with $ \gamma $ as in
(\ref{HardyEstimate}).   It is also straightforward 
to  see that the estimate (\ref{HardyEstimate}) implies that   $ \sum _ { j
  =1 } ^ \infty \lambda _j v _j$ defines an 
element of the Hardy-Sobolev  space $H^ { 1,p } ( \partial \Omega)$. 
Now we assume $ p < 1$, fix $\phi \in C ^ \alpha ( \bar \Omega)$
with $ \alpha = \frac 1 p - 1$,  and consider the limit
\begin{equation*}
\begin{split} 
& \left | \lim _ { k \rightarrow \infty } \int _ { \partial \Omega _k }
\frac { \partial u }{ \partial \rho } \phi \, d\sigma 
- \sum _ { j = 1} ^ \infty \lambda _j \int _ { \partial \Omega } \frac
{ \partial v _j }{ \partial \rho } \phi \, d\sigma \right | \\
&\qquad \qquad \leq \limsup _ { k \rightarrow \infty }  \sum _ { j =1 } ^ M
 \lambda _ j  \left |  \int _ { \partial \Omega _ k } \frac { \partial v _ j }{
  \partial \rho } \phi \, d\sigma 
- \int _ { \partial \Omega  } \frac { \partial v _ j }{
  \partial \rho } \phi \, d\sigma   \right | 
\\
& \qquad \qquad \qquad + \sup _ k \sum _ { j = M+ 1} ^ \infty \lambda _j( | \int _ { \partial
  \Omega_k  } \frac { \partial v _ j }{ \partial \rho } \phi \, d\sigma  | +
| \int _ { \partial
  \Omega  } \frac { \partial v _ j }{ \partial \rho } \phi \, d\sigma  | ).
\end{split}
\end{equation*}
Using the estimate  (\ref{HardyEstimate}) for the non-tangential
maximal function of $ \nabla 
v _j $   and the dominated convergence theorem, we  see that the
first term on 
 the right of this inequality is zero for any $M$. 
Since functions in $ C^ \alpha $, $ \alpha = 1/p -1$ give rise to
elements of  the dual of $H^p$, 
the second
and third terms are small when $M$ is large by our choice of  $
\lambda _j$ and the estimate for $ \partial v_j / \partial \rho$ in 
 (\ref{pEx1}). Hence we have that $ \partial u / \partial \rho $
exists in the sense of (\ref{HPBV}) and is given by $ \sum \lambda _j
\partial v_j / \partial \rho$ and we have the estimate  for $ \partial
u / \partial \rho$ in the estimate of the Theorem. Note that 
we also obtain  the estimate when $p=1$. 
 Finally, if we restrict $
\phi$ to lie in $ C^ \alpha _D ( \bar  \Omega )$, we have that $
\partial v _j / \partial \rho  = a_j$ on the set $N$ and thus we have
$ \partial u / \partial \rho = f _N$. 
\end{proof}

Finally, we  observe that a real-variable argument of Caffarelli and Peral
\cite{MR1486629} and Shen \cite{ZS:2007} gives existence for $1 < p <
p _0$. The argument to obtain this result is identical to that used in
our study of the Laplacian in our earlier work with Taylor
\cite{OB:2009,TOB:2011}.

\begin{theorem} 
\label{LpEx}
Let $ \Omega$ be a Lipschitz domain and let  $D \subset \partial
\Omega$  be  a
non-empty proper open subset  of $ \partial \Omega$ which satisfies
(\ref{Corkscrew}), (\ref{BIntegral}), and (\ref{Integral}).  Suppose
that $ 1< p< p_0$ with $p_0$ as in Theorem \ref{Hardy}. We assume
that the coefficients of the operator $L$ satisfy (\ref{Elliptic}). 
Let $ f _N \in L^ p ( N)$ and $ f_D =0$, then we may find a solution of the
$L^p$-mixed problem (\ref{MP}) which satisfies 
$$
\| \nontan { ( \nabla u )} \|_ { L ^ p ( \partial \Omega)} \leq C \| f
_ N \| _ { L ^ p ( N)}.
$$
\end{theorem}

The following lemma is a key estimate 
 that is needed to carry out the method of Shen \cite{ZS:2007}. 
This result  may be proved using the techniques in the proof of Theorem
\ref{Hardy}. See also the argument in section 6 of our previous work
\cite{OB:2009}. 

\begin{lemma}
\label{NT}  
Let $ \Omega$, $D$, and $L$ be as in Theorem \ref{LpEx}. 
Suppose that $ 1 <p < p _0 = q_0 (1-\epsilon) / ( 2- \epsilon)$. 

If $u$ is a weak solution of (\ref{WMP}) with data $ f_N   $ in $ L^ p
( N)$ and $f=0$, we have the following local estimate 
for $ 1<  p < p _ 0 $,
\begin{equation}
\label{NT1} 
\left ( \average _ {\sball x r }  { \nontan {( \nabla u ) _ {r} }}^ p \,
d\sigma \right ) ^ { 1/ p }  \leq C ( \average _ { \locdom x { 4r} }
|\nabla u | \, dy + ( \average   _ { \sball x { 2 r} \cap N }
|f_N | ^ p \, d\sigma ) ^ { 1/p } ) .
\end{equation}
The constant $C$ depends only on $M$ and the exponent $p$. 
\end{lemma}

\note{
\begin{tabular}{ll}
$d$ & distance to boundary\\
$\delta$ & distance to $\Lambda$ \\
$q_0$ & upper bound for integrability in $ \partial \Omega$\\
\end{tabular}
}

\appendix
\renewcommand*{\thesection}{\Alph{section}}
\renewcommand{\theequation}{\Alph{section}.\arabic{equation}}

\section{Sobolev inequalities}
\label{Inequalities}

In this appendix, we establish the estimates (\ref{SoPo2}) and
(\ref{BoPo}) 
 that  were used   in the study of the mixed problem. 
These results may be found in  our earlier work 
\cite{OB:2009,TKB:2012,TOB:2011}, though  we 
do not claim originality. The exposition below serves to collect these
results in one place and includes an occasional endpoint that was
missed in our earlier work. 

Let $\phi$ be a Lipschitz function on the unit sphere so that $ \Omega = \{ y : |y | < r \phi(
y / |y |) \}$ is a star-shaped  Lipschitz domain  of scale $r$ and
centered at 0.  
We  define a 
bi-Lipschitz map $ \Phi : \ball 0  1  \rightarrow \Omega$ by 
\begin{equation}
\label{BiL}
\Phi(y ) =  r \phi(y/|y|) y .
\end{equation}
Using the argument in \cite[Lemma 7.16]{GT:1983} and a change of
variables, we may find a constant $C= C(N,n,p)$ so that for $ S
\subset \Omega$, a set of positive Lebesgue measure and $ 1 \leq p < \infty$, we have
\begin{equation}
\label{Poincare2} 
\| u -  \average _S u \, dy \|_ { L^ p ( \Omega) } 
\leq \frac  { C(M,n,p) r^ { n+1} } { |S| } \| \nabla u \|_ { L^ p ( \Omega)}.
\end{equation}
The only change that is
needed from the  standard argument is to average $u$ with respect to
the weight given by the Jacobian of the change of variables. 

Next we observe that for $ \Omega$ a star-shaped convex domain with
constant $M$ and scale $r$, $ 1 \leq p < n$,  $q$ defined by  $1/q = 1/p - 1/n $, and $S$ a measurable subset of $  \Omega$,  we have
the Sobolev-Poincar\'e inequality
\begin{equation}
\label{SoPo} 
\| u - \average _S u \, dy \|_ { L^ q ( \Omega) } 
\leq \frac  { C(M,p,n) r^ { n} } { |S| } \| \nabla u \|_ { L^ p ( \Omega)}.
\end{equation}
To establish (\ref{SoPo}) we extend $u$ to $ \reals ^ n$ by a reflection in $
\partial \Omega$.  Thus choose $\eta$ a cutoff function which is one
when $ |x| \leq ( 1+ M) r $ and zero  when $ |x| > ( 2+ 2M) r$. 
We let 
$$
Eu(x) = \left \{ \begin{array}{ll} 
u(x) , \qquad &  x \in \Omega \\
 u( x r^2 \phi (x/|x|)^2/ |x|^ 2) \eta (x)   , \qquad & x \in \reals ^ n
\setminus \Omega.
\end{array}
\right. 
$$
\note{ Note that the map $ V(x) =x  r ^ 2 \phi( x/|x|)^2 /
  |x|^2 $ is locally Lipschitz, fixes the boundary of $\Omega$, the
  set $ |x| = r \phi(x/|x|) $, and satisfies $ \Phi\circ \Phi= I$. 
}
The Sobolev inequality and the  product rule in $ \reals ^ n$ 
gives 
\begin{eqnarray*}
\| u - \average _S u\, dy  \| _ { L^ q (  \Omega) }
& \leq &  \| E( u- \average_S u \, dy )\| _ { L^ q ( \reals ^ n ) }  \\
& \leq &  C( p, n ) \| \nabla E ( u - \average _S u \, dy ) \| _ { L^ p ( \reals
  ^ n ) } \\
& \leq  & C ( \| \nabla u \|_ { L^ p ( \Omega) } + r^ { -1} \| u - \average
_S u \, dy  \| _ { L^ p ( \Omega) } ) .
\end{eqnarray*}
Applying the estimate  (\ref{Poincare2}) gives the Sobolev-Poincar\'e
inequality (\ref{SoPo}).  

The following estimate was proved in \cite{OB:2009} under an
additional assumption that the  set $D$ was not too spread out. The proof
below is simpler and omits that assumption. 

\begin{lemma}
\label{Wolff} 
 If $ \Omega$ is a  star-shaped Lipschitz domain with
  scale $r$ and constant $M$, $ D $ a measurable subset of  $\partial
  \Omega$,  and  $ u \in \sobolev p 1 _D ( \Omega)$, then 
  for  $ 1 \leq p < n $ and $ q$ given by $ 1/q = 1/p - 1/n$, we have 
$$
\left ( \int _  \Omega |u |^ {q } \, dy \right ) ^ { 1/q } 
\leq C \frac { r ^ { n -1 } } { \sigma (D )} \left ( \int _ \Omega
|\nabla u | ^ p \, dy \right ) ^ { 1/p } 
$$
The constant $C$ depends on $n$, $p$, and $M$. 
\end{lemma}
\begin{proof}  
It suffices to prove this estimate for a  function $u$ in $ C^ \infty
( \bar \Omega)$ 
which vanishes in a neighborhood of $D$. Using the map $ \Phi$ defined
in (\ref{BiL}), we may reduce to considering the domain $ \Omega =    
\ball 0 1 $. We set $ A = \{ y \in \partial \ball 0 1 : \Phi (y)
\in D \} $, then we have 
$$
C(M) \sigma (D) \leq \sigma (A) r^ { n-1} \leq \sigma (D) .
$$
We will show that if $ u \in C^ \infty ( \bar  B _1 (0)  )$  and $u$ 
vanishes on $A \subset \partial \ball  0 1 $, then for $ 1\leq p < n $ 
\begin{equation}
\label{Claim} 
\left ( \int  _ { \ball 0 1 } |u |^ { np/ ( n - p ) } \, dy \right) ^ { 1/p -
  1/n } \leq \frac C { \sigma (A) } \left ( \int _ { \ball 0 1 }
|\nabla u  | ^ p \, dy \right ) ^ { 1/p } . 
\end{equation}
As the Lemma follows from  the special case where $ \Omega$ is  the
unit ball, $ \ball 0 1 $,
we only need to prove the claim (\ref{Claim}). 

We let $ \tilde A
= \{ y \in \ball 0 1 : y / |y| \in A \}$ and  observe that 
\begin{equation}
\label{Step1} 
\left | \int _ { \tilde A } u(y) \, dy \right | 
\leq \frac 1 n \int _ { \tilde A } |y \cdot \nabla u(y)|  \, dy.
\end{equation}
To establish (\ref{Step1}), we let $s $ be in $[0,1]$, $y \in A$ and
 use the fundamental theorem of calculus to write
$$
u(sy) = -\int _ s ^1 \frac d { dt} u(ty) \, dt  = - \int _ s ^ 1 y \cdot
\nabla u( ty) \, dt .
$$
We multiply by $s^ { n -1}$ and  integrate on $ [0,1] \times A$ to obtain
$$
\int _ { \tilde A } u( y) \, dy =  \int _A \int _ 0 ^ 1 u(sy) s^ {
  n-1}\,d s \, d\sigma  = 
-\frac 1 n \int _ { \tilde A } z \cdot \nabla u (z) \, dz
$$
which implies (\ref{Step1}). Given (\ref{Step1}), the claim
(\ref{Claim}) follows from (\ref{SoPo}).
\end{proof}

We are now ready to 
 establish (\ref{SoPo2}). Observe that if $ \dist( \locdom x \rho , D) >
0$, then $ \avg u x \rho = \average _ { \locdom x \rho } u \, dy $ and
thus (\ref{SoPo2}) follows from (\ref{SoPo}) and  in this case,  the
constant remains bounded as $ \rho $ approaches $r$.  If $ \avg u x
\rho =0$, then we have $ \dist ( \locdom x \rho  , D ) = 0$. In this
case, the corkscrew condition (\ref{Corkscrew}) implies that $ \sigma
(D \cap \bar \Omega _r (x) ) \geq c ( r - \rho ) ^ { n -1} $ and
(\ref{SoPo2}) follows from Lemma \ref{Wolff}.

To establish the inequality (\ref{BoPo}) we work in a coordinate
cylinder and 
let $ \eta (x', x_n ) $ be a cut off function which is 1 for $
|x_n - \phi ( x')| < \rho
/2$ and 0 for $  |x_n - \phi ( x')| > \rho$. We apply the 
 fundamental theorem of calculus and obtain that 
\begin{equation}
\label{BoPoGreen}
-  \int _ { \sball x \rho } |u- \avg u x \rho |^ q e_n \cdot \nu \, d\sigma = 
 \int _ { \locdom x \rho } q  \frac { ( u -
   \avg u  x \rho   ) ^ \alpha } { | u - \avg u x \rho |}
  \frac { \partial u ^ \alpha  } { \partial y _n }
 |u- \avg u x \rho|^ { q-1} \eta + |u - \avg u x\rho | ^ q
 \frac{\partial \eta }{\partial y_n} \, dy.  
\end{equation}
We apply the H\"older inequality and obtain that the following is an
upper bound to the right-hand side of (\ref{BoPoGreen}),  
$$
 q\left ( \int _ {\locdom x \rho } \left | \frac { \partial u }{\partial y _n } \right| ^ { p } 
 	\, dy\right)^ { 1/p } 
\left ( \int _ { \locdom  x \rho } |u - \avg u x \rho | ^ { (q-1)p'}\, dy \right) ^ { 1/p'} 
 + C \left( \int _ { \locdom x \rho } |u -\avg u x\rho |^ { q n / ( n -1 )} \, 
 dy \right ) ^ { (n-1)/n } . 
$$
 
If we have the relation $ (q-1) p' = np/(n-p)$, or $q = p(n-1) / ( n-p)$, 
we may use (\ref{SoPo2}) to obtain
$$
\left ( \int _ { \locdom x \rho } |u - \avg u x \rho |^ { (q-1)p'} \, dy \right)  ^ 
{ 1/( p'(q-1))}
\leq C \frac { r ^ { n-1}}{ ( r- \rho ) ^ { n -1 }}
\left( \int _ { \locdom x r } |\nabla u |^ { p} \, dy \right) ^ {
  1/p}. 
$$
If we have the relation $qn/ ( n-1) = np/(n-p) $, then the estimate
(\ref{SoPo2}) will give us 
$$
\left( \int _ {\locdom x \rho } |u- \avg u x \rho|^ {qn/ ( n-1)} \, dy \right ) ^ { 1/p} 
\leq C \frac { r ^ { n-1}}{ ( r- \rho ) ^ { n-1}}
 \left( \int _ { \locdom x  r  } |\nabla u |^ p \, \right) ^ {(n-1)/nq }.
$$
Simplifying gives that  $ q$ and $p$ are related by $ 1/p= 1/2 + 1/(2q) $ 
 if $n =2$
or $1/p = 1/n + (n-1)/(nq)$ in general, 
which gives (\ref{BoPo}).

\section{The regularity problem in $H^ { 1, p }( \partial \Omega)$.}
\label{RegularityProblem}

The goal of this appendix  is to treat the $L^p$-regularity problem for the
Lam\'e operator $L$ when the data lies in the Hardy-Sobolev space $H^{
  1,p}( \partial \Omega)$ (see section \ref{Functions} for the
definition of this space).  By the $L^p$-regularity problem, we mean
the following boundary value problem
\begin{equation}
\label{RegPro}
\left\{ \begin{array}{ll} 
Lu =0, \qquad & \mbox{in } \Omega  \\
u=f , \qquad & \mbox{on } \partial \Omega \\
\nontan {( \nabla u ) } \in L^ p ( \partial \Omega). 
  \end{array}
\right.
\end{equation}
The boundary values  are taken in the sense of non-tangential
limits. The argument we give is adapted from an argument  of    Pipher and
Verchota \cite{PV:1992}  used to study boundary value problems for the
bi-harmonic operator in three
dimensions and was subsequently used to study  the traction problem and
regularity problem for the Lam\'e system by
Dahlberg and Kenig \cite{DK:1990} for $p>1$.   Our main result is the following
theorem which treats the $L^p$-regularity problem
 in two and three dimensions  and  $p  \leq 1$.  
\begin{theorem}
\label{RPEx}
Let $\Omega \subset \reals ^ n$, with $n=2$ or $3$, be a Lipschitz
domain and suppose that $L$ is the Lam\'e operator with coefficients
satisfying (\ref{Elliptic}). 

There exists $ p_1<1$ so that regularity problem (\ref{RegPro}) has a
solution for $ p_1 < p \leq  1 $. More precisely we have:

If $f$ is in $H^ { 1,p} ( \partial \Omega)$, $p_1< p \leq 1$, then
there exists a solution of the $ L^p$-regularity problem which
satisfies
$$
\| \frac { \partial u } { \partial \rho } \| _ { H ^ p ( \partial
  \Omega )} + \| \nontan { ( \nabla u )} \|_ {L^ p ( \partial \Omega)}  
\leq C \| f \| _ { H ^ { 1, p }( \partial \Omega)}.
$$

The constant in this estimate depends on $M$ and the global character
of $ \Omega$. 

Furthermore, if $ p>p_1$ and  $u$ is a solution of the
$L^p$-regularity problem,  (\ref{RegPro}), 
 with $f =0$, then $u =0$. 
\end{theorem}

We recall several facts that will be useful in our proof. 
A  real-variable estimate of Dindo\v  s and Mitrea  \cite[Lemma
  6.1]{MR1934200}   tells us that
there is a constant $C$ so that 
\begin{equation}
\label{DiMi}
\left( \int _ \Omega |u | ^ p \, d y  \right) ^ { 1/p } 
\leq C \left ( \int _ { \partial \Omega  } (\nontan u )^ { p(n-1)/n} \,
d\sigma \right ) ^ { n / p ( n-1)}, \qquad  p > 0.
\end{equation}
In work from 1988, Dahlberg, Kenig, and Verchota \cite{DKV:1988}
treated the Dirichlet problem for the Lam\'e system (also see work of
Mayboroda and Mitrea if $n =2$ \cite{MR2181934}). From this  work 
 and a real-variable argument we can show
that  there
exists $\epsilon >0$ so that the $L^p$-Dirichlet problem
\begin{equation} 
\label{LpDp}
\left\{ 
\begin{array}{ll} 
Lu =0 ,  \qquad & \mbox{in } \Omega \\
u =f , \qquad & \mbox{on } \partial \Omega \\
\nontan u \in L^ p ( \partial \Omega )
\end{array}
\right.
\end{equation}
has a unique solution for $p$ in the range $ (2- \epsilon , 2 +
\epsilon)$. This solution  satisfies the estimate $ \| \nontan u \| _ { L^ p (
  \partial \Omega )} \leq C \| f \| _ { L^ p ( \partial \Omega )}$. 

\comment{
We introduce some notation that will be needed in the  next Lemma.
Given a  boundary ball $ \sball x r $, we 
let $ \Sigma _ 0 = \{ y : |y -x| < r \}$ and $ \Sigma _k =
\{ y : 2 ^ { k-1} r \leq |x-y  | < 2 ^ k r \}$. }  The following
technical result is the main step in our argument.  The proof follows
ideas of Pipher and Verchota and especially Dahlberg and Kenig
\cite[Lemma 1.6]{DK:1990}.  This Lemma will use the sets $\Sigma _k$
as defined in section 4. 

\begin{lemma}
\label{RPDecay} 
Suppose $ \Omega$ is a Lipschitz domain in $ \reals ^ n$, $n =2, 3$, 
and $L$ is the Lam\'e operator with coefficients satisfying the
conditions (\ref{Elliptic}). If $ A$ is a 1-atom supported in a
boundary ball $ \sball x r$, then there
exists $ \gamma >0$ so that
\begin{equation}
\label{RP1} 
\left ( \int  _ { \Sigma _k } |\nabla u |^ 2 \, d\sigma \right ) ^
{ 1/2 } \leq C 2 ^ { - \gamma k } ( 2^ k r ) ^ {-( n-1)/2 }    , 
  \qquad k =0,1,2,\dots. 
\end{equation}
and for $  1\geq   p > 1/(1+\gamma) $, 
$$
\| \nontan {( \nabla u )} \| _ { L^ p ( \partial \Omega)} \leq C
\sigma ( \sball x r  ) ^ { \frac 1 p -1}.
$$
\end{lemma}

\begin{proof} We let $ A$ be a 1-atom supported in a boundary ball $
  \sball x r$. From the work of Dahlberg, Kenig, and Verchota
  \cite[Theorem 3.7]{DKV:1988} (see also Mayboroda and Mitrea  \cite{MR2181934} for  results
  in two dimensions), we have a solution of the regularity problem
  (\ref{RegPro}) with $\nontan{(\nabla u )} \in L^2 ( \partial \Omega
  )$ that  satisfies the estimate 
$$
\| \nontan { ( \nabla u ) } \|_ {L^ 2 (\partial \Omega )} \leq C \|
\nabla _t  A \| _{ L^ 2 ( \sball x r )} \leq C \sigma ( \sball xr )^
       { -1/2}.
$$
This implies the estimate (\ref{RP1}) with a constant that depends on
$k$. Thus, it suffices  to prove (\ref{RP1}) for $k \geq k_0$. 
As in the proof of Lemma \ref{Decay}  we choose $ k _0$ so that $ \dist (\sball
x r , \Sigma_ { k _0 } ) \approx r$. 

To establish (\ref{RP1}) for $k \geq k_0$,  fix $ R> r$ and suppose that $ \sball y { 8 R}$ is a
boundary ball with $ \sball y { 8 R} \cap \sball x r = \emptyset $. We
use the results Lemma \ref{LocalSolve} and (\ref{Cacc}) to obtain 
\begin{equation}
\label{RP2}
\int _ { \sball y R } |\nabla u | ^ 2 \, d\sigma \leq \frac  C R \int  _ {
  \locdom y {2R} } |\nabla u |^ 2\, dz
\leq \frac  C { R^ 3} \int _{ \locdom y { 4 R}} |u | ^ 2\,dz  .
\end{equation}
We let $ E(y,R) = \{ z\in \partial \Omega : \Gamma (z) \cap \locdom y
{ 4R } \neq \emptyset \}$ and observe that $E(y,R) \subset \{ z : |y -
z | \leq CR \}$ where $C = C(M, \alpha )$. We fix an exponent $ t $ in
$ ( 2-\epsilon , 2 )$ such that we can solve the $ L^t$-Dirichlet
problem. 
The estimate of Dindo\v s and Mitrea (\ref{DiMi}), H\"older's inequality and the
solvability of the $L^t$-Dirichlet problem imply  that we have
\begin{equation}
\label{RP3} 
\begin{split}
\frac 1 { R ^ 3} \int _ { \locdom y {4R}} |u | ^2 \, dz 
& \leq \frac C {
  R ^ 3} \left( \int _ {E(y,R)} ( \nontan u ) ^ { 2 ( n-1)/n} \,d
\sigma \right) ^ { ( n-1) /n} \\
& \leq C R ^ { n-3 - ( n-1) 2/t} \| \nontan u \|^2_{ L^ { t  } (
  \partial \Omega )} 
\leq C  \| A \|^2 _ { L^ t ( \sball x r )} R ^ { n-3 - (n-1) 2 /t} \\
&\leq C r ^ { ( n-1) 2 /t + 4 -2n} R ^ { n-3 - ( n-1) 2 /t } .
\end{split}
\end{equation}
We have used the properties of the atom $A$ to obtain that $ |A|\leq C
/r^ { n-2} $ in the last line. 
If we combine (\ref{RP2}) and (\ref{RP3}) 
we obtain 
\begin{equation}
\label{RP4}
\int _ { \sball  y R } |\nabla u |^2 \, d\sigma \leq  C \left ( \! \frac r
R \! \right) ^ { 3-n - ( n -1 ) ( 2 /t -1)} \sigma ( \sball y R ) ^ {
  -1} .
\end{equation}
We let $\gamma  =
\frac 1 2 ( ( n -3)   + ( n -1) ( \frac 2 t -1))$. Since $ t < 2$,
$ \gamma $ will be positive if $ n  \leq 3$.  The estimate
(\ref{RP1}) follows from (\ref{RP4})  by an elementary covering
argument. 

The second estimate follows using a  representation formula for the
gradient of a solution $u$ as in the
proof of Theorem \ref{Hardy}. Note that it is much easier to justify
the integration by parts  in this case 
as we have $\nontan {( \nabla   u )} \in L^ 2 ( \partial \Omega )$. 
\end{proof}
\note{Several exponents corrected in (\ref{RP3}) and (\ref{RP4}).}

Now we are ready to give the proof of the main result of this
appendix. 
\begin{proof}[Proof of Theorem \ref{RPEx}]
Let $ f= \sum _ j \lambda _j A _j$ with each $ A _j$ a 1-atom
supported in $ \Delta _j = \sball { x_j } {r _j }$ and $ \sum _ { j }
\lambda _ j ^ p \sigma ( \Delta _ j ) ^ { 1-p} < \infty$. 
We let $v_j$
be the solution of (\ref{RegPro})  with $ f = A_j$ as given in the
work of Dahlberg, Kenig, and Verchota \cite{DKV:1988} and consider the
sum $ u = \sum _ j \lambda _j v _j$. Using the estimate (\ref{DiMi})
and the Poincar\'e inequality, we have that as long as $p> ( n-1) /n$, 
$$
\| v_j \| _ { L^ { pn/(n-1)}  ( \Omega)} \leq C \| \nabla v _ j \| _ { L^ {
      pn/(n-1)}( \Omega)} \leq C \| \nontan { ( \nabla v _j )} \| _ {
    L^ p ( \partial \Omega) }  \leq C \sigma ( \Delta _j ) ^ { \frac 1
    p -1}. 
$$
The elementary inequality  $ (\sum a _ j ) ^p \leq \sum a_j ^p$,  valid if 
$a_j \geq 0$ and $ p<1$, implies that for $M < M'$, we have 
$$
\| \sum _ { j = M}^ { M ' } \lambda _ j v _ j \| _ { L^ { pn/(n-1) }(
  \Omega)}
\leq \left ( \sum _ { j =  M } ^ {M'} \lambda _j ^ p \sigma ( \Delta
_j ) ^ { 1- p } \right ) ^ { 1/ p } 
$$
which implies that the series defining $u$ converges at least in $ L^
{ pn / ( n -1 ) }( \Omega)$. The estimate for $\| \nontan {( \nabla
  v_j )} \| _ { L^ p ( \partial \Omega)}$  in Lemma \ref{RPDecay}
implies that 
$$
\| \nontan {( \nabla u )} \| _ { L^ p ( \partial \Omega) } \leq C 
\left ( \sum _ j \lambda _ j ^ p \sigma ( \Delta _j )  ^ { 1-p} \right
) ^ { 1/p } , \qquad p > 1/( 1+ \gamma ) 
$$
and that   $\nabla u$ has non-tangential limits a.e. in $ \partial
\Omega$. 
We may follow the argument in Theorem \ref{pEx} to show that $\partial
u / \partial \rho $ lies in $H^ p ( \partial \Omega)$. 
This  completes the proof of existence of solutions to the
$L^p$-regularity problem for $p > \max \{ ( n-1)/n, 1/( 1+
\gamma) \} $.

We establish the uniqueness of solutions to the $L^p$-regularity
problem. Suppose $u$ is a solution of (\ref{RegPro}) with $u=0$
a.e. on $ \partial \Omega$. From Lemma \ref{FI} or the result in
Mitrea and Wright \cite[Lemma 11.9]{MW:2011} for $n =3$, we obtain that $\nontan u \in
L^ q ( \partial \Omega )$, $ 1 /q = 1/p - 1/ ( n-1)$. Thus, if $ q >
2-\epsilon$ with $2 -\epsilon$ the lower bound for uniqueness in the
Dirichlet problem from the work of Dahlberg, Kenig, and Verchota
\cite{DKV:1988}, we may conclude that $u=0$. In two dimensions, this
gives uniqueness if $ p > 1 - 1/ ( 3 - \epsilon)$ and in three
dimensions we obtain uniqueness if $ p > 1 - \epsilon / ( 4 -
\epsilon)$.
\end{proof}

\bibliographystyle{plain}
\def\cprime{$'$} \def\cprime{$'$} \def\cprime{$'$} \def\cprime{$'$}

\medskip 

\noindent{\small \today}
\end{document}